\documentclass[12pt,reqno]{amsart}
\usepackage{amsmath,amsfonts,amssymb,amsthm,amsbsy,latexsym,cite,amscd}
\usepackage[cp1251]{inputenc}
\usepackage[english]{babel}
\usepackage[%
 colorlinks,
 breaklinks,
 linkcolor=blue,
 citecolor=blue,
 menucolor=blue,
 pagecolor=blue,
]{hyperref}

\multlinegap=\parindent

\begin{document}

\newtheorem{rtheorem}{Теорема}
\newtheorem{rcorollary}{Следствие}
\newtheorem{rproposition}{Предложение}[section]
\newtheorem{etheorem}{Theorem}
\newtheorem{ecorollary}{Corollary}
\newtheorem{eproposition}{Proposition}[section]

\theoremstyle{definition}
\newtheorem{rexample}{Пример}
\newtheorem{eexample}{Example}

\renewcommand{\thefootnote}{}

\title[On~the~residual nilpotence of~generalized free products]{On~the~residual nilpotence of~generalized~free~products~of~groups}

\author{E.~V.~Sokolov}
\address{Ivanovo State University, Russia}
\email{ev-sokolov@yandex.ru}

\begin{abstract}
Let $G$ be~the~generalized free product of~two groups with~an~amalgamated subgroup. We propose an~approach that allows one to~use results on~the~residual $p$\nobreakdash-fi\-nite\-ness of~$G$ for~proving that this generalized free product is~residually a~finite nilpotent group or~residually a~finite metanilpotent group. This approach can be~applied under~most of~the~conditions on~the~amalgamated subgroup that allow the~study of~residual $p$\nobreakdash-fi\-nite\-ness. Namely, we consider the~cases where the~amalgamated subgroup is~a)~periodic, b)~locally cyclic, c)~central in~one of~the~free factors, d)~normal in~both free factors, or~e)~is~a~retract of~one of~the~free factors. In~each of~these cases, we give certain necessary and~sufficient conditions for~$G$ to~be~residually a)~a~finite nilpotent group, b)~a~finite metanilpotent group.
\end{abstract}


\thanks{The~study was supported by~the~Russian Science Foundation grant No.~22-21-00166,\\ \url{https://rscf.ru/en/project/22-21-00166/}}

\maketitle\vspace{-12pt}

\section{Introduction}\label{es01}

Let $\mathcal{C}$ be~a~class of~groups. Following~\cite{Hall1954PLMS}, we say that a~group~$X$ is~\emph{residually a~$\mathcal{C}$\nobreakdash-group} if,~for~any element $x \in X \setminus \{1\}$, there exists a~homomorphism~$\sigma$ of~$X$ onto~a~group from~$\mathcal{C}$ (a~\emph{$\mathcal{C}$\nobreakdash-group}) such that $x\sigma \ne 1$. The~terms ``\emph{residual finiteness}'', ``\emph{residual nilpotence}'', ``\emph{residual solvability}'', and~``\emph{residual $p$\nobreakdash-fi\-nite\-ness}'' are~used if~$\mathcal{C}$ is~the~class of~all finite groups, all nilpotent groups, all solvable groups, and~finite $p$\nobreakdash-groups for~some prime~$p$, respectively. In~this article, we study the~residual $\mathcal{C}$\nobreakdash-ness of~a~generalized free product of~two groups provided $\mathcal{C}$ is~the~class of~finite nilpotent or~finite metanilpotent groups.

Recall that a~group is~said to~be~\emph{metanilpotent} if~it is~an~extension of~a~nilpotent group by~a~nilpotent group. The~definition of~the~construction of~the~generalized free product of~two groups with~an~amalgamated subgroup can be~found in~Section~\ref{es05}. It~should also be~noted that a~finitely generated group is~residually nilpotent if and~only~if~it is~residually a~finite nilpotent group (see Proposition~\ref{ep38} below).

\medskip

Let us preface the~description of~the~results of~this article with~a~brief survey of~known facts on~the~residual nilpotence of~generalized free products. Since, for~any choice of~prime number~$p$, every finite $p$\nobreakdash-group is~nilpotent, residual $p$\nobreakdash-fi\-nite\-ness is~a~special case of~residual nilpotence. Most of~the~known results on~the~residual nilpotence of~generalized free products actually give necessary and/or~sufficient conditions only for~the~residual $p$\nobreakdash-fi\-nite\-ness of~this construction. This is~explained by~the~fact that the~class of~finite $p$\nobreakdash-groups is~closed under~extensions. Thanks to~this property, to~prove the~residual $p$\nobreakdash-fi\-nite\-ness of~a~generalized free product, one can use very productive methods, which were originally proposed for~studying the~residual finiteness of~this construction (primarily we mean here the~so-called ``filtration approach'' by~G.~Baumslag~\cite{BaumslagG1963TAMS}).

A~criterion for~the~residual $p$\nobreakdash-fi\-nite\-ness of~a~generalized free product of~two finite groups is~given in~\cite{Higman1964JA}. In~the~case of~infinite free factors, residual $p$\nobreakdash-fi\-nite\-ness is~studied mainly under~various additional restrictions imposed on~the~amalgamated subgroup. For~example, in~\cite{Azarov1997SMJ, Azarov2006BIvSU, Azarov2017IVM}, a~number of~theorems are~proved that generalize the~criterion from~\cite{Higman1964JA} and~give conditions for~the~residual $p$\nobreakdash-fi\-nite\-ness of~a~generalized free product with~a~finite amalgamation. Much effort is~put into~studying the~residual $p$\nobreakdash-fi\-nite\-ness of~free products with~a~cyclic amalgamated subgroup~\cite{Gildenhuys1975JAMS, KimMcCarron1993JA, Doniz1996JA, KimTang1998JA, Azarov1998MN, WongTangGan2006RMJM, Azarov2016SMJ} and~with an~amalgamated subgroup that is~normal in~every free factor~\cite{Sokolov2005MN, KimLeeMcCarron2008KM, Tumanova2014MN, Azarov2015SMJ, Tumanova2015IVM, SokolovTumanova2020IVM}. A~free product is~also considered whose amalgamated subgroup lies in~the~center of~at~least one free factor~\cite{SokolovTumanova2023SMJ} or~is~a~retract of~this factor~\cite{BobrovskiiSokolov2010AC, SokolovTumanova2016SMJ}. In~particular, the~following criteria are~proved:

--\hspace{1ex}a~criterion for~the~residual $p$\nobreakdash-fi\-nite\-ness of~a~generalized free product of~two isomorphic groups~\cite[Corollary~3.5]{KimMcCarron1993JA};

--\hspace{1ex}a~criterion for~the~residual $p$\nobreakdash-fi\-nite\-ness of~a~generalized free product whose amalgamated subgroup is~a~retract of~each free factor~\cite[Corollary~1.2]{BobrovskiiSokolov2010AC};

--\hspace{1ex}criteria for~the~residual $p$\nobreakdash-fi\-nite\-ness of~a~generalized free product with~a~cyclic amalgamation of~a) two free groups~\cite[Theorem~1]{Azarov1998MN}, \cite[Theorem~3.5]{KimTang1998JA}; b) two nilpotent groups of~finite rank~\cite[Theorem~2]{Azarov2016SMJ};

--\hspace{1ex}criteria for~the~residual $p$\nobreakdash-fi\-nite\-ness of~a~generalized free product such that its amalgamated subgroup is normal in~every free factor and~at~least one of~the~following additional conditions holds: a)~each free factor is~virtually a~solvable group of~finite rank~\cite[Theorem~2]{Azarov2015SMJ}; b)~the~amalgamated subgroup is~finite~\cite[Theorem~3]{Tumanova2014MN}; c)~one of~the~free factors is~a~finitely generated nilpotent group, and~the~amalgamated subgroup lies in~its center~\cite[Theorem~4.9]{KimLeeMcCarron2008KM} or~is~cyclic~\cite[Theorem~5]{Tumanova2014MN}.

Recall that a~group is~said to~have \emph{finite rank~$r$} if~any of~its finitely generated subgroups can be~generated by~at~most $r$~elements.

\medskip

If~$\mathcal{C}$ is~a~class of~nilpotent groups that is~not closed under~taking extensions, there are~no~general methods for~studying the~residual $\mathcal{C}$\nobreakdash-ness of~generalized free products. As~a~consequence, very few results are~known in~this case. Only the~following is~proved:{\parfillskip=0pt\par}

--\hspace{1ex}a~criterion for~the~residual nilpotence of~a~generalized free product of~two finite groups~\cite[Theorem~6]{Varsos1996HJM} and~its generalization to~the~case of~a~free product of~two finitely generated nilpotent groups with~a~finite amalgamation~\cite[Theorem~4]{Ivanova2004PhD};

--\hspace{1ex}a~criterion for~the~residual nilpotence of~a~generalized free product of~two nilpotent groups with~an~amalgamated subgroup lying in~the~center of~each free factor~\cite[Theorem~1]{RozovSokolov2016BTSUMM};

--\hspace{1ex}three criteria for~the~residual nilpotence of~a~generalized free product of~two finitely generated nilpotent groups; their description is~given at~the~end of~the~next section~\cite[Theorems~5,~7 and~Corollary of~Theorem~6]{Ivanova2004PhD};

--\hspace{1ex}a~criterion for~a~generalized free product of~two nilpotent groups to~be~residually a~finite nilpotent group, which holds provided the~amalgamated subgroup is~central in~one and~normal in~the~other free factor~\cite[Theorem~2]{RozovSokolov2016BTSUMM};

--\hspace{1ex}several sufficient conditions for~a~generalized free product with~a~cyclic amalgamation to~be~residually a~tor\-sion-free nilpotent group; the~cases are~considered when the~free factors are~free groups~\cite{BaumslagGCleary2006JGT, BaumslagG2010IJM, BaumslagGMikhailov2014GGD, Labute2015Arxiv} or~one of~these factors is~a~free group and~the~other is~a~free abelian group~\cite{BaumslagG1962MZ}.

It is~proved also that, with~a~suitable choice of~a~prime~$p$, the~residual nilpotence is~equivalent to~the~residual $p$\nobreakdash-fi\-nite\-ness for~the~following constructions:

--\hspace{1ex}a~generalized free product of~two free groups with~a~cyclic amalgamation~\cite[Theorem~3]{Azarov1998MN};

--\hspace{1ex}a~generalized free product of~two finitely generated tor\-sion-free nilpotent groups whose amalgamated subgroup is~cyclic or~lies in~the~center of~each free factor~\cite[Corollary of~Theorem~6]{Ivanova2004PhD}.

\medskip
\enlargethispage{4pt}

This paper significantly complements the~above list. It~proposes a~general approach that allows one to~use results on~the~residual $p$\nobreakdash-fi\-nite\-ness of~a~generalized free product~$G$ to~study the~residual $\mathcal{C}$\nobreakdash-ness of~$G$, where $\mathcal{C}$ is~the~class of~finite nilpotent or~finite metanilpotent groups. We show that this approach can be~applied under~each of~the~restrictions on~the~amalgamated subgroup, which are~mentioned above and~allow the~study of~residual $p$\nobreakdash-fi\-nite\-ness. The~results obtained in~this way are~formulated in~the~next section, while here we give a~number of~concepts and~notations used throughout the~paper.

Suppose that $\mathfrak{P}$ is~a~set of~prime numbers, $X$~is~a~group, and~$Y$ is~a~subgroup of~$X$. Recall that an~integer is~said to~be~a~\emph{$\mathfrak{P}$\nobreakdash-num\-ber} if~all its prime divisors belong to~$\mathfrak{P}$; a~periodic group is~said to~be~a~\emph{$\mathfrak{P}$\nobreakdash-group} if~the~orders of~all its elements are~$\mathfrak{P}$\nobreakdash-num\-bers. We denote by~$\mathfrak{P}^{\prime}$ the~set of~all prime numbers that do~not belong to~$\mathfrak{P}$.

The~subgroup~$Y$ is~called \emph{$\mathfrak{P}^{\prime}$\nobreakdash-iso\-lat\-ed} in~$X$ if,~for~any element $x \in X$ and~for~any number $q \in \mathfrak{P}^{\prime}$, it~follows from~the~inclusion $x^{q} \in Y$ that $x \in Y$. If~the~trivial subgroup of~$X$ is~$\mathfrak{P}^{\prime}$\nobreakdash-iso\-lat\-ed, then $X$ is~said to~be~a~\emph{$\mathfrak{P}^{\prime}$\nobreakdash-tor\-sion-free group}. It~is~easy to~see that the~intersection of~any number of~$\mathfrak{P}^{\prime}$\nobreakdash-iso\-lat\-ed subgroups is~again a~$\mathfrak{P}^{\prime}$\nobreakdash-iso\-lat\-ed subgroup and~therefore one can define the~smallest $\mathfrak{P}^{\prime}$\nobreakdash-iso\-lat\-ed subgroup containing the~subgroup~$Y$. It~is~called the~\emph{$\mathfrak{P}^{\prime}$\nobreakdash-iso\-la\-tor} of~$Y$ in~$X$ and~is~further denoted by~$\mathfrak{P}^{\prime}\textrm{-}\mathfrak{Is}(X,Y)$. Obviously, $\mathfrak{P}^{\prime}\textrm{-}\mathfrak{Is}(X,Y)$ contains the~subset~$\mathfrak{P}^{\prime}\textrm{-}\mathfrak{Rt}(X,Y)$ of~$X$ such that $x \in \mathfrak{P}^{\prime}\textrm{-}\mathfrak{Rt}(X,Y)$ if and~only~if~$x^{n} \in Y$ for~some $\mathfrak{P}^{\prime}$\nobreakdash-num\-ber~$n$. It~is~clear that the~equality $\mathfrak{P}^{\prime}\textrm{-}\mathfrak{Is}(X,Y) = \mathfrak{P}^{\prime}\textrm{-}\mathfrak{Rt}(X,Y)$ holds if and~only~if~the~set~$\mathfrak{P}^{\prime}\textrm{-}\mathfrak{Rt}(X,Y)$ is~a~subgroup.{\parfillskip=0pt\par}

Following~\cite{Malcev1958PIvSU}, we say that the~subgroup~$Y$ is~\emph{separable in~$X$ by~the~class of~groups~$\mathcal{C}$} (or,~more briefly, is~\emph{$\mathcal{C}$\nobreakdash-sep\-a\-ra\-ble in~$X$}) if,~for~any element $x \in X \setminus Y$, there exists a~homomorphism~$\sigma$ of~$X$ onto~a~group from~$\mathcal{C}$ such that $x\sigma \notin Y\sigma$. As~can be~seen from~this definition, the~residual $\mathcal{C}$\nobreakdash-ness of~$X$ is~equivalent to~the~$\mathcal{C}$\nobreakdash-sep\-a\-ra\-bil\-ity of~its trivial subgroup. It~is~also easy to~prove that if~the~subgroup~$Y$ is~separable in~$X$ by~the~class of~finite $\mathfrak{P}$\nobreakdash-groups, then it is~$\mathfrak{P}^{\prime}$\nobreakdash-iso\-lat\-ed in~$X$ (see Proposition~\ref{ep32} below).

Let us call an~abelian group \emph{$\mathfrak{P}$\nobreakdash-bound\-ed} if,~in each of~its quotient groups, a~primary component of~the~periodic part is~finite whenever it corresponds to~a~number from~$\mathfrak{P}$. We say also that a~nilpotent group is~\emph{$\mathfrak{P}$\nobreakdash-bound\-ed} if~it has a~finite central series with~$\mathfrak{P}$\nobreakdash-bound\-ed abelian factors. The~classes of~$\mathfrak{P}$\nobreakdash-bound\-ed abelian and~$\mathfrak{P}$\nobreakdash-bound\-ed nilpotent groups are~further denoted by~$\mathcal{BA}_{\mathfrak{P}}$ and~$\mathcal{BN}_{\mathfrak{P}}$, respectively. Obviously, if~$\mathfrak{S} \subseteq \mathfrak{P}$, then $\mathcal{BA}_{\mathfrak{P}} \subseteq \mathcal{BA}_{\mathfrak{S}}$ and~$\mathcal{BN}_{\mathfrak{P}} \subseteq \mathcal{BN}_{\mathfrak{S}}$. It~is~also easy to~see that a~finitely generated nilpotent group is~$\mathfrak{P}$\nobreakdash-bound\-ed for~any choice of~the~set~$\mathfrak{P}$.

Everywhere below let~$\Phi$, $\mathcal{F}_{\mathfrak{P}}$, and~$\mathcal{FN}_{\mathfrak{P}}$ stand for~the~classes of~free groups, finite $\mathfrak{P}$\nobreakdash-groups, and~finite nilpotent $\mathfrak{P}$\nobreakdash-groups, respectively. If~$\mathfrak{P} = \{p\}$, we simply write $p$ instead of~$\{p\}$ (for~example: $p$\nobreakdash-num\-ber, $\mathcal{BN}_{p}$, $p^{\prime}\textrm{-}\mathfrak{Is}(X,Y)$). We also use a~number of~standard notations:

\smallskip

\makebox[8ex][l]{$|X|$}the~order of~a~group~$X$;

\makebox[8ex][l]{$[x,y]$}the~commutator of~elements~$x$ and~$y$, which is~assumed to~be~equal to~the~prod- \makebox[8ex]{}\hspace{\parindent}uct $x^{-1}y^{-1}xy$;

\makebox[8ex][l]{$\operatorname{sgp}\{\kern-1pt{}S\}$}the~subgroup generated by~a~set $S$;

\makebox[8ex][l]{$X^{\prime}$}the~commutator subgroup of~$X$, i.e.,~$\operatorname{sgp}\{[x,y] \mid x,y \in X\}$;

\makebox[8ex][l]{$X^{n}$}$\operatorname{sgp}\{x^{n} \mid x \in X\}$;

\makebox[8ex][l]{$\ker\sigma$}the~kernel of~a~homomorphism~$\sigma$;

\makebox[8ex][l]{$\mathcal{C} \kern1pt{\cdot}\kern1pt \mathcal{D}$}the\kern-.5pt{}~class\kern-.5pt{} of\kern-.5pt{}~groups\kern-.5pt{} consisting\kern-.5pt{} of\kern-.5pt{}~all\kern-.5pt{} possible\kern-.5pt{} extensions\kern-.5pt{} of\kern-.5pt{}~a\kern-.5pt{}~$\mathcal{C}$\nobreakdash-group\kern-.5pt{} by\kern-.5pt{}~a\kern-.5pt{}~$\mathcal{D}$\nobreakdash-group, \makebox[8ex]{}\hspace{\parindent}where $\mathcal{C}$ and~$\mathcal{D}$ are~some classes of~groups.

\section{New results}\label{es02}

Let further the~expression $G = \langle A * B;\ H\rangle$ mean that $G$ is~the~generalized free product of~groups~$A$ and~$B$ with~an~amalgamated subgroup~$H$ (as~noted above, the~definition and~necessary properties of~this construction are~given in~Section~\ref{es05}). We say that the~group $G = \langle A * B;\ H\rangle$ and~a~set of~primes~$\mathfrak{P}$ \pagebreak \emph{satisfy Condition~$(*)$}~if

\makebox[4.5ex][l]{$(\mathfrak{i})$}$A \ne H \ne B$;

\makebox[4.5ex][l]{$(\mathfrak{ii})$}$\mathfrak{P}$ is~a~non-empty~set;

\makebox[4.5ex][l]{$(\mathfrak{iii})$}$A$~and~$B$ are~residually $\mathcal{FN}_{\mathfrak{P}}$\nobreakdash-groups;

\makebox[4.5ex][l]{$(\mathfrak{iv})$}there\kern-.5pt{} exist\kern-.5pt{} homomorphisms\kern-.5pt{} of\kern-.5pt{}~the\kern-.5pt{}~groups\kern-.5pt{}~$A$\kern-.5pt{} and\kern-.5pt{}~$B$\kern-.5pt{} that\kern-.5pt{} map\kern-.5pt{} them\kern-.5pt{} onto\kern-.5pt{}~$\mathcal{BN}_{\mathfrak{P}}$\nobreakdash-groups and~act injectively~on~$H$.

Almost all the~main results of~this paper deal with~generalized free products satisfying~$(*)$. Parts~$(\mathfrak{i})$---$(\mathfrak{iii})$ of~this condition look quite natural because a)~if~$G$ is~residually an~$\mathcal{FN}_{\mathfrak{P}}$\nobreakdash-group, then $A$ and~$B$ have the~same property; b)~if $\mathfrak{P} = \varnothing$, only~the~trivial group is~residually an~$\mathcal{FN}_{\mathfrak{P}}$\nobreakdash-group; c)~if $A = H$ or~$B = H$, then $G = B$ or~$G = A$ and~it follows from~$(\mathfrak{iii})$ that $G$ is~residually an~$\mathcal{FN}_{\mathfrak{P}}$\nobreakdash-group. The~more artificial restriction~$(\mathfrak{iv})$ is~necessary to~obtain the~results mentioned above via~Theorem~\ref{et07}. This theorem is~given in~Section~\ref{es08} and~allows one to~prove the~residual nilpotence of~$G$ using known facts on~the~residual $p$\nobreakdash-fi\-nite\-ness of~this group.

It is~easy to~see that if~$H$ is~finite, then $(\mathfrak{iii})$ implies~$(\mathfrak{iv})$. In~the~case of~the~infinite subgroup~$H$, homomorphisms satisfying~$(\mathfrak{iv})$ certainly exist if~$A$ and~$B$ are~residually tor\-sion-free $\mathcal{BN}_{\mathfrak{P}}$\nobreakdash-groups and~$H$ is~of~finite Hirsch--Zaitsev rank, i.e.,~it~has a~finite subnormal series whose each factor is~a~periodic or~infinite cyclic group (see Proposition~\ref{ep35} below). The~following groups are~examples of~residually tor\-sion-free $\mathcal{BN}_{\mathfrak{P}}$\nobreakdash-groups:

--\hspace{1ex}free~\cite{Magnus1935MA}, parafree~\cite{BaumslagG1967TAMS}, limit~\cite{Sela2001PM}, and~free polynilpotent~\cite{Gruenberg1957PLMS} groups;

--\hspace{1ex}fundamental groups of~two-di\-men\-sion\-al orientable closed manifolds~\cite{BaumslagG1962MZ};

--\hspace{1ex}pure braid groups of~various types~\cite{FalkRandell1988CM, Markushevich1991GD, BellingeriGervaisGuaschi2008JA, BardakovBellingeri2009CA, Marin2012CM};

--\hspace{1ex}right-an\-gled Artin groups~\cite{DuchampKrob1992SF} and~their Torelli groups~\cite{Toinet2013GGD};

--\hspace{1ex}Hydra groups and~some other groups with~one defining relation, including the~cyclically pinched one-re\-la\-tor groups mentioned in~the~previous section~\cite{BaumslagGCleary2006JGT, BaumslagG2010IJM, BaumslagGMikhailov2014GGD, Labute2015Arxiv};

--\hspace{1ex}all restricted and~unrestricted direct products of~the~groups listed above.

The~following theorem is~the~first of~the~main results of~this article.

\begin{etheorem}\label{et01}
Suppose that the~group $G = \langle A * B;\ H\rangle$ and~a~set of~primes~$\mathfrak{P}$ satisfy~$(*)$ and~at~least one of~the~next additional conditions\textup{:}

\makebox[4ex][l]{$(\alpha)$}$H$~is~locally cyclic\textup{;}

\makebox[4ex][l]{$(\beta\kern.3pt)$}$H$~lies in~the~center of~$A$~or~$B$\textup{;}

\makebox[4ex][l]{$(\gamma\kern1pt)$}$H$~is~a~retract of~$A$~or~$B$.

\noindent
Then the~following statements hold.

\textup{1.}\hspace{1ex}If~$H$ is~$\mathcal{F}_{q}$\nobreakdash-sep\-a\-ra\-ble in~$A$ and~$B$ for~some $q \in \mathfrak{P}$\textup{,} then $G$ is~residually an~$\mathcal{FN}_{\mathfrak{P}}$\nobreakdash-group.

\textup{2.}\hspace{1ex}If~$H$ is~$\mathcal{FN}_{\mathfrak{P}}$\nobreakdash-sep\-a\-ra\-ble in~$A$ and~$B$\textup{,} then $G$ is~residually a~$\Phi \kern1pt{\cdot}\kern1pt \mathcal{FN}_{\mathfrak{P}}$\nobreakdash-group and~residually an~$\mathcal{F}_{p} \kern1pt{\cdot}\kern1pt \mathcal{FN}_{\mathfrak{P}}$\nobreakdash-group for~each prime~$p$. In~particular\textup{,} it~is~residually a~finite metanilpotent $\mathfrak{P}$\nobreakdash-group.
\end{etheorem}

If~$X$ is~a~group and~$Y$ is~its normal subgroup, then the~restriction of~any inner automorphism of~$X$ onto~the~subgroup~$Y$ is~an~automorphism of~the~latter. The~set of~all such automorphisms forms a~subgroup of~the~group~$\operatorname{Aut}Y$, which we further denote by~$\operatorname{Aut}_{X}(Y)$. If~$G = \langle A * B;\ H\rangle$ and~the~subgroup~$H$ is~normal in~$A$ and~$B$, then it is~also normal in~$G$ and~this allows us to~define the~group~$\operatorname{Aut}_{G}(H)$. It~turns~out that the~properties of~the~latter can be~used to~describe the~conditions for~the~residual $\mathcal{C}$\nobreakdash-ness of~$G$; this approach was first used in~\cite{Higman1964JA} and~then repeatedly in~other works (see, for~example,~\cite{Tumanova2014MN, Tumanova2015IVM}). Let us note also that if~$H$ is~normal in~$G$, then, for~each prime number~$p$, the~subgroup~$H^{p}H^{\prime}$ has the~same property and~therefore the~groups~$\mathfrak{A}(p) = \operatorname{Aut}_{A/H^{p}H^{\prime}}(H/H^{p}H^{\prime})$, $\mathfrak{B}(p) = \operatorname{Aut}_{B/H^{p}H^{\prime}}(H/H^{p}H^{\prime})$, and~$\mathfrak{G}(p) = \operatorname{Aut}_{G/H^{p}H^{\prime}}(H/H^{p}H^{\prime})$ are~defined.

\begin{etheorem}\label{et02}
Suppose that the~group $G = \langle A * B;\ H\rangle$ and~a~set of~primes~$\mathfrak{P}$ satisfy~$(*)$. Suppose also that $H$ is~normal in~$A$ and~$B$ and\textup{,}~for~each $p \in \mathfrak{P}$\textup{,} at~least one of~the~next additional conditions holds\textup{:}

\makebox[4ex][l]{$(\alpha)$}$\mathfrak{G}(p)$~is~a~$p$\nobreakdash-group\textup{;}

\makebox[4ex][l]{$(\beta\kern.3pt)$}$\mathfrak{G}(p)$~is~abelian\textup{;}

\makebox[4ex][l]{$(\gamma\kern1pt)$}$\mathfrak{G}(p) = \mathfrak{A}(p)$ or~$\mathfrak{G}(p) = \mathfrak{B}(p)$.

\noindent
Then the~following statements hold.

\textup{1.}\hspace{1ex}If~$H$ is~$\mathcal{F}_{q}$\nobreakdash-sep\-a\-ra\-ble in~$A$ and~$B$ for~some $q \in \mathfrak{P}$\textup{,} then $G$ is~residually an~$\mathcal{FN}_{\mathfrak{P}}$\nobreakdash-group.

\textup{2.}\hspace{1ex}If~$G$ is~residually an~$\mathcal{FN}_{\mathfrak{P}}$\nobreakdash-group\textup{,} then $H$ is~$q^{\prime}$\nobreakdash-iso\-lat\-ed in~$A$ and~$B$ for~some $q \in \mathfrak{P}$.

\textup{3.}\hspace{1ex}The~group~$G$ is~residually a~$\Phi \kern1pt{\cdot}\kern1pt \mathcal{FN}_{\mathfrak{P}}$\nobreakdash-group and~residually an~$\mathcal{F}_{p} \kern1pt{\cdot}\kern1pt \mathcal{FN}_{\mathfrak{P}}$\nobreakdash-group for~each prime~$p$ if and~only~if~$H$ is~$\mathcal{FN}_{\mathfrak{P}}$\nobreakdash-sep\-a\-ra\-ble in~$A$ and~$B$.
\end{etheorem}

It should be~noted that Theorem~\ref{et02} is~a~special case of~Theorem~\ref{et06}. The~formulation of~the~latter requires the~preliminary proof of~some auxiliary facts and~is~therefore given in~Section~\ref{es05}. We note also that if~the~subgroup~$H/H^{p}H^{\prime}$ lies in~the~center, say, of~the~group~$A/H^{p}H^{\prime}$, then $\mathfrak{A}(p) = 1$ and~$\mathfrak{G}(p) = \mathfrak{B}(p)$ because $\mathfrak{G}(p)$ is~obviously generated by~$\mathfrak{A}(p)$ and~$\mathfrak{B}(p)$. If~the~indicated subgroup is~locally cyclic, then its automorphism group is~abelian (see, for~example,~\cite[\S~113, Exercise~4]{Fuchs21973}). Thus, the~next corollary follows from~Theorem~\ref{et02}.

\begin{ecorollary}\label{ec01}
Suppose that the~group $G = \langle A * B;\ H\rangle$ and~a~set of~primes~$\mathfrak{P}$ satisfy~$(*)$. Suppose also that $H$ is~normal in~$A$ and~$B$ and\textup{,}~for~each $p \in \mathfrak{P}$\textup{,} the~subgroup~$H/H^{p}H^{\prime}$ is~locally cyclic or~lies in~the~center of~at~least one of~the~groups~$A/H^{p}H^{\prime}$\textup{,} $B/H^{p}H^{\prime}$. Then Statements~\textup{1\nobreakdash---3} of~Theorem~\textup{\ref{et02}} hold.
\end{ecorollary}

\begin{etheorem}\label{et03}
Suppose that the~group $G = \langle A * B;\ H\rangle$ and~a~set of~primes~$\mathfrak{P}$ satisfy~$(*)$. Suppose also that $H$ is~periodic and\textup{,}~for~every prime~$p$\textup{,} the~symbol~$H(p)$ denotes the~subgroup~$p\textrm{-}\mathfrak{Is}(\kern-.5pt{}H\kern-1.5pt{},\kern-.5pt{}1)$\kern-.5pt{}.\kern-1pt{} Then $H$\kern-1pt{} is~$\mathcal{F\kern-.5pt{}N}\kern-.5pt{}_{\mathfrak{P}}$\nobreakdash-sep\-a\-ra\-ble in~$A$\kern-.5pt{} and~$B$\kern-.5pt{}\textup{,}\kern-1pt{} and~the~\mbox{following}~\mbox{statements}~hold.

\textup{1.}\hspace{1ex}The~group~$G$ is~residually an~$\mathcal{FN}_{\mathfrak{P}}$\nobreakdash-group if~$H$ is~$\mathcal{F}_{q}$\nobreakdash-sep\-a\-ra\-ble in~$A$ and~$B$ for~some $q \in \mathfrak{P}$ and\textup{,}~for~any $p \in \mathfrak{P}$\textup{,} there exist sequences of~subgroups
\begin{gather*}
1 = Q_{0} \leqslant Q_{1} \leqslant \ldots \leqslant Q_{n} = H(p),\\
R_{0} \leqslant R_{1} \leqslant \ldots \leqslant R_{n} = A,
\quad
S_{0} \leqslant S_{1} \leqslant \ldots \leqslant S_{n} = B
\end{gather*}
such that\textup{,} for~all $i \in \{0,\,1,\,\ldots,\,n\}$\textup{,} $j \in \{0,\,1,\,\ldots,\,n-1\}$\textup{,}

\makebox[1ex]{}\phantom{1.}--\hspace{1ex}$R_{i}$~is~a~normal subgroup of~finite $p$\nobreakdash-in\-dex of~the~group~$A$\textup{;}

\makebox[1ex]{}\phantom{1.}--\hspace{1ex}$S_{i}$\kern1.1pt~is~a~normal subgroup of~finite $p$\nobreakdash-in\-dex of~the~group~$B$\textup{;}

\makebox[1ex]{}\phantom{1.}--\hspace{1ex}$R_{i} \cap H(p) = Q_{i} = S_{i} \cap H(p)$\textup{;}

\makebox[1ex]{}\phantom{1.}--\hspace{1ex}$|Q_{j+1}/Q_{j}| = p$.

\textup{2.}\hspace{1ex}If~$G$ is~residually an~$\mathcal{FN}_{\mathfrak{P}}$\nobreakdash-group\textup{,} then $H$ is~$q^{\prime}$\nobreakdash-iso\-lat\-ed in~$A$ and~$B$ for~some $q \in \mathfrak{P}$ and\textup{,}~for~any $p \in \mathfrak{P}$\textup{,} there exist the~sequences of~subgroups described in~Statement~\textup{1} of~this theorem.

\textup{3.}\hspace{1ex}The~group~$G$ is~residually a~$\Phi \kern1pt{\cdot}\kern1pt \mathcal{FN}_{\mathfrak{P}}$\nobreakdash-group and~residually an~$\mathcal{F}_{q} \kern1pt{\cdot}\kern1pt \mathcal{FN}_{\mathfrak{P}}$\nobreakdash-group for~each prime~$q$ if and~only~if\textup{,}~for~any $p \in \mathfrak{P}$\textup{,} there exist the~sequences of~subgroups described in~Statement~\textup{1} of~this theorem.
\end{etheorem}

Let $\tau A$ and~$\tau B$ be~the~sets of~all elements of~finite order of~the~groups~$A$ and~$B$, respectively. If~the~indicated groups have homomorphisms onto~$\mathcal{BN}_{\mathfrak{P}}$\nobreakdash-groups acting injectively~not only on~$H$, but~also on~the~subgroups~$\operatorname{sgp}\{\tau A\}$ and~$\operatorname{sgp}\{\tau B\}$, then Theorem~\ref{et03} admits the~next equivalent formulation generalizing Theorem~4 from~\cite{Azarov2017IVM}.

\begin{etheorem}\label{et04}
Suppose that the~group $G = \langle A * B;\ H\rangle$ and~a~set of~primes~$\mathfrak{P}$ satisfy~$(*)$. Suppose also that $H$ is~periodic and~there exist homomorphisms of~the~groups~$A$ and~$B$ that map them onto~$\mathcal{BN}_{\mathfrak{P}}$\nobreakdash-groups and~act injectively on~the~subgroups~$\operatorname{sgp}\{\tau A\}$ and~$\operatorname{sgp}\{\tau B\}$. Then\textup{,} for~each $p \in \mathfrak{P}$\textup{,} the~sets $A(p) = p\textrm{-}\mathfrak{Rt}(A,1)$ and~$B(p) = p\textrm{-}\mathfrak{Rt}(B,1)$ are~finite normal subgroups of~$A$ and~$B$\textup{,} respectively\textup{,} $H$~is~$\mathcal{FN}_{\mathfrak{P}}$\nobreakdash-sep\-a\-ra\-ble in~$A$ and~$B$\textup{,} and~the~following statements hold.

\textup{1.}\hspace{1ex}The~group~$G$ is~residually an~$\mathcal{FN}_{\mathfrak{P}}$\nobreakdash-group if~$H$ is~$\mathcal{F}_{q}$\nobreakdash-sep\-a\-ra\-ble in~$A$ and~$B$ for~some $q \in \mathfrak{P}$ and\textup{,}~for~any $p \in \mathfrak{P}$\textup{,} the~groups~$A(p)$ and~$B(p)$ have normal series
$$
1 = A_{0} \leqslant A_{1} \leqslant \ldots \leqslant A_{k} = A(p)
\quad\text{and}\quad
1 = B_{0} \leqslant B_{1} \leqslant \ldots \leqslant B_{m} = B(p)
$$
with factors of~order~$p$ such that\textup{,} for~all $i \in \{0,\,1,\,\ldots,\,k\}$\textup{,} $j \in \{0,\,1,\,\ldots,\,m\}$\textup{,}

\makebox[1ex]{}\phantom{1.}--\hspace{1ex}$A_{i}$~is~a~normal subgroup of~$A$\textup{;}

\makebox[1ex]{}\phantom{1.}--\hspace{1ex}$B_{j}$~is~a~normal subgroup of~$B$\textup{;}

\makebox[1ex]{}\phantom{1.}--\hspace{1ex}$\{A_{i} \cap H \mid 0 \leqslant i \leqslant k\} = \{B_{j} \cap H \mid 0 \leqslant j \leqslant m\}$.

\textup{2.}\hspace{1ex}If~$G$ is~residually an~$\mathcal{FN}_{\mathfrak{P}}$\nobreakdash-group\textup{,} then $H$ is~$q^{\prime}$\nobreakdash-iso\-lat\-ed in~$A$ and~$B$ for~some $q \in \mathfrak{P}$ and\textup{,}~for~any $p \in \mathfrak{P}$\textup{,} there exist normal series described in~Statement~\textup{1} of~this theorem.{\parfillskip=0pt\par}

\textup{3.}\hspace{1ex}The~group~$G$ is~residually a~$\Phi \kern1pt{\cdot}\kern1pt \mathcal{FN}_{\mathfrak{P}}$\nobreakdash-group and~residually an~$\mathcal{F}_{q} \kern1pt{\cdot}\kern1pt \mathcal{FN}_{\mathfrak{P}}$\nobreakdash-group for~each prime~$q$ if and~only~if\textup{,}~for~any $p \in \mathfrak{P}$\textup{,} the~groups~$A(p)$ and~$B(p)$ have normal series described in~Statement~\textup{1}.
\end{etheorem}

Let us now turn to~the~question on~the~conditions that are~necessary for~a~generalized free product to~be~residually a~nilpotent or~metanilpotent group.

\begin{etheorem}\label{et05}
Suppose that $G = \langle A * B;\ H\rangle$\textup{,} $\mathfrak{P}$ is~a~non-empty set of~primes\textup{,} and~$H$ is~nilpotent. Suppose also that

\makebox[4ex][l]{$(\alpha)$}$H$~does~not coincide with~its normalizer in~$A$ or~is~properly contained in~some subgroup of~$A$ satisfying a~non-triv\-i\-al identity\textup{;}

\makebox[4ex][l]{$(\beta\kern.3pt)$}$H$~does~not coincide with~its normalizer in~$B$ or~is~properly contained in~some subgroup of~$B$ satisfying a~non-triv\-i\-al identity.

Then the~following statements hold.

\textup{1.}\hspace{1ex}If~$G$ is~residually an~$\mathcal{FN}_{\mathfrak{P}}$\nobreakdash-group\textup{,} then $H$ is~$q^{\prime}$\nobreakdash-iso\-lat\-ed in~$A$ and~$B$ for~some $q \in \mathfrak{P}$.

\textup{2.}\hspace{1ex}If~$G$ is~residually a~$\mathcal{C}$\nobreakdash-group for~some subclass~$\mathcal{C}$ of~the~class~$\mathcal{F}_{\mathfrak{P}}$ and~$\mathcal{C}$ is~closed under~taking subgroups\textup{,} then $H$ is~$\mathcal{C}$\nobreakdash-sep\-a\-ra\-ble and~therefore $\mathfrak{P}^{\prime}$\nobreakdash-iso\-lat\-ed in~$A$ and~$B$.
\end{etheorem}

The~next corollary is~obtained by~combining Theorems~\ref{et01},~\ref{et05} and~Proposition~\ref{ep91} given below.

\begin{ecorollary}\label{ec02}
Suppose that the~group $G = \langle A * B;\ H\rangle$ and~a~set of~primes~$\mathfrak{P}$ satisfy~$(*)$\textup{,} Conditions~$(\alpha)$ and~$(\beta)$ of~Theorem~\textup{\ref{et05},} and~at~least one of~Conditions~$(\alpha)$\nobreakdash---$(\gamma)$ of~Theorem~\textup{\ref{et01}}. Then Statements~\textup{1\nobreakdash---3} of~Theorem~\textup{\ref{et02}} hold.
\end{ecorollary}

The~following example shows, in~particular, that, in~Theorem~\ref{et05}, none of~Conditions~$(\alpha)$ and~$(\beta)$ can be~omitted.

\begin{eexample}\label{ex01}
Suppose that $\mathfrak{P}$ is~a~non-empty set of~prime numbers, which does~not coincide with~the~set of~all primes, and~$G = \langle a_{1}, a_{2}, b;\ a_{2} = b^{n}\rangle$, where $n \in \mathfrak{P}^{\prime}$. Then $G$ is~the~generalized free product of~the~non-abelian free group~$A = \operatorname{sgp}\{a_{1}, a_{2}\}$ and~the~infinite cyclic group~$B = \operatorname{sgp}\{b\}$ with~the~amalgamated cyclic subgroup~$H = \operatorname{sgp}\{b^{n}\}$, which lies in~the~center of~$B$ and~is~a~retract of~$A$. At~the~same time, the~elements~$a_{1}$ and~$b$ freely generate the~group~$G$ and~therefore the~latter is~residually an~$\mathcal{FN}_{\mathfrak{P}}$\nobreakdash-group (see Proposition~\ref{ep38} below). It~is~also obvious that the~retraction of~$A$ and~the~identity homomorphism of~$B$ act injectively on~$H$ and~map~$A$ and~$B$ onto~infinite cyclic groups, which belong to~the~class~$\mathcal{BN}_{\mathfrak{P}}$. Thus, $G$ satisfies all the~conditions of~Theorem~\ref{et01}. At~the~same time, $H$~is~neither $\mathfrak{P}^{\prime}$\nobreakdash-iso\-lat\-ed nor~$\mathcal{FN}_{\mathfrak{P}}$\nobreakdash-sep\-a\-ra\-ble in~$B$. Also, it~is~neither $q^{\prime}$\nobreakdash-iso\-lat\-ed nor~$\mathcal{F}_{q}$\nobreakdash-sep\-a\-ra\-ble in~this group for~any $q \in \mathfrak{P}$.
\end{eexample}

\begin{eexample}\label{ex02}
Let
$$
G = \big\langle a,b,c;\ a^{9} = [a^{3},b] = [a^{3},c] = c^{-1}bcb = 1\big\rangle
$$
and~$\mathfrak{P} = \{2, 3\}$. It~is~obvious that

\textup{1)}\hspace{1ex}$G$~is~the~generalized free product of~the~groups~$A = \operatorname{sgp}\{a\}$ and~$B = \operatorname{sgp}\{a^{3},b,c\}$ with~the~subgroup~$H = \operatorname{sgp}\{a^{3}\}$ amalgamated;

\textup{2)}\hspace{1ex}$A$~is~a~finite cyclic group of~order~$9$ and~therefore belongs to~the~class~$\mathcal{F}_{3} \subseteq \mathcal{BN}_{\mathfrak{P}}$;{\parfillskip=0pt\par}

\textup{3)}\hspace{1ex}$B$~can be~decomposed into~the~direct product of~the~subgroups~$H$ and~$C = \operatorname{sgp}\{b, c\}$; moreover, $C$ is~the~non-abelian extension of~the~infinite cyclic group~$\operatorname{sgp}\{b\}$ by~the~infinite cyclic group~$\operatorname{sgp}\{c\}$;

\textup{4)}\hspace{1ex}the~finite cyclic group~$H$ lies in~the~center of~$G$ and~is~a~retract of~$B$; the~identity homomorphism of~$A$ and~the~retraction of~$B$ act injectively on~$H$ and~map~$A$ and~$B$ onto~groups from~the~class~$\mathcal{F}_{3} \subseteq \mathcal{BN}_{\mathfrak{P}}$;

\textup{5)}\hspace{1ex}if\kern-.5pt{}~we\kern-.5pt{} define\kern-.5pt{} the\kern-.5pt{}~group\kern-.5pt{}~$\mathfrak{G}(p)$\kern-.5pt{} in\kern-.5pt{}~the\kern-.5pt{}~same\kern-.5pt{} way\kern-.5pt{} as\kern-.5pt{}~in\kern-.5pt{}~Theorem\kern-.5pt{}~\ref{et02},\kern-.5pt{} then\kern-.5pt{} $\mathfrak{G}(3)\kern-1.5pt{} \cong\kern-3.5pt{} \operatorname{Aut}_{G}(\kern-1pt{}H)\kern-1.5pt{} =\nolinebreak\kern-2pt{} 1$ and~$\mathfrak{G}(2) = \operatorname{Aut}_{G/H}(1) = 1$;

\textup{6)}\hspace{1ex}the~following sequences satisfy the~conditions of~Theorem~\ref{et03}: $1 = Q_{0} = H(p)$, $R_{0} = A$, and~$S_{0} = B$ if~$p = 2$; $1 = Q_{0} \leqslant Q_{1} = H(p) = H$, $1 = R_{0} \leqslant R_{1} = A$, and~$C = S_{0} \leqslant S_{1} = B$ if~$p = 3$;

\textup{7)}\hspace{1ex}the~following normal series satisfy the~conditions of~Theorem~\ref{et04}: $1 = A_{0} = A(p)$ and~$1 = B_{0} = B(p)$ if~$p = 2$; $1 = A_{0} \leqslant H \leqslant A = A(p)$ and~$1 = B_{0} \leqslant H = B(p)$ if~$p = 3$.{\parfillskip=0pt\par}

It is~well known (see, for~example,~\cite[Theorem~2]{Moldavanskii2018CA}) that $C$ is~residually an~$\mathcal{F}_{2}$\nobreakdash-group. Therefore, $B$ is~residually an~$\mathcal{FN}_{\mathfrak{P}}$\nobreakdash-group and~$G$ satisfies all the~conditions of~Theorems~\ref{et01}\nobreakdash---\ref{et04} and~Corollaries~\ref{ec01},~\ref{ec02}. Moreover, since the~group~$C \cong B/H$ is~tor\-sion-free, $H$~is~isolated in~$B$ and~therefore $3^{\prime}$\nobreakdash-iso\-lat\-ed in~$A$ and~$B$. However, $G$ is~not residually an~$\mathcal{FN}_{\mathfrak{P}}$\nobreakdash-group by~Proposition~\ref{ep62} given below.
\end{eexample}

Examples~\ref{ex01} and~\ref{ex02} imply the~following.

\textup{1.}\hspace{1ex}In~Theorem~\ref{et01}, the~$\mathcal{FN}_{\mathfrak{P}}$\nobreakdash-sep\-a\-ra\-bil\-ity of~$H$ in~$A$ and~$B$ is, in~the~general case, not~necessary for~$G$ to~be~residually an~$\mathcal{F}_{q} \kern1pt{\cdot}\kern1pt \mathcal{FN}_{\mathfrak{P}}$\nobreakdash-group, where $q \in \mathfrak{P}$; the~fact that $G$ is~residually an~$\mathcal{FN}_{\mathfrak{P}}$\nobreakdash-group does~not necessarily mean that $H$ is~$\mathcal{F}_{q}$\nobreakdash-sep\-a\-ra\-ble or~even $q^{\prime}$\nobreakdash-iso\-lat\-ed in~$A$ and~$B$ for~some $q \in \mathfrak{P}$.

\textup{2.}\hspace{1ex}In~Theorems~\ref{et02}\nobreakdash---\ref{et04} and~Corollaries~\ref{ec01},~\ref{ec02}, the~property of~$H$ to~be~$q^{\prime}$\nobreakdash-iso\-lat\-ed in~$A$ and~$B$ for~some $q \in \mathfrak{P}$ is~necessary, but~not~sufficient for~$G$ to~be~residually an~$\mathcal{FN}_{\mathfrak{P}}$\nobreakdash-group. The~question on~the~necessity of~the~$\mathcal{F}_{q}$\nobreakdash-sep\-a\-ra\-bil\-ity of~$H$ remains open.

\medskip

Let us note that $\mathcal{BN}_{\mathfrak{P}} \subseteq \mathcal{BN}_{q}$ for~any $q \in \mathfrak{P}$ and~therefore, by~Proposition~\ref{ep44} given below, a~subgroup of~a~$\mathcal{BN}_{\mathfrak{P}}$\nobreakdash-group~$X$ is~$\mathcal{F}_{q}$\nobreakdash-sep\-a\-ra\-ble ($\mathcal{FN}_{\mathfrak{P}}$\nobreakdash-sep\-a\-ra\-ble) in~this group if and~only~if~it is~$q^{\prime}$\nobreakdash-iso\-lat\-ed (respectively, $\mathfrak{P}^{\prime}$\nobreakdash-iso\-lat\-ed) in~$X$. Since every $\mathcal{BN}_{\mathfrak{P}}$\nobreakdash-group also satisfies a~non-triv\-i\-al identity, the~next corollary follows from~Theorems~\ref{et01}\nobreakdash---\ref{et05}.

\begin{ecorollary}\label{ec03}
Suppose that $G = \langle A * B;\ H\rangle$\textup{,} $\mathfrak{P}$ is~a~non-empty set of~primes\textup{,} $A$~and~$B$ are~$\mathfrak{P}^{\prime}$\nobreakdash-tor\-sion-free $\mathcal{BN}_{\mathfrak{P}}$\nobreakdash-groups\textup{,} and~$A \ne H \ne B$. Then the~following statements hold.{\parfillskip=0pt\par}

\textup{1.}\hspace{1ex}If~at~least one of~the~conditions~$(\alpha)$\nobreakdash---$(\gamma)$ of~Theorem~\textup{\ref{et01}} is~satisfied\textup{,} then $G$ is~residually an~$\mathcal{FN}_{\mathfrak{P}}$\nobreakdash-group \textup{(}an~$\mathcal{FN}_{\mathfrak{P}} \kern1pt{\cdot}\kern1pt \mathcal{FN}_{\mathfrak{P}}$\nobreakdash-group\textup{)} if and~only~if~$H$ is~$q^{\prime}$\nobreakdash-iso\-lat\-ed in~$A$ and~$B$ for~some $q \in \mathfrak{P}$ \textup{(}respectively\textup{,} $H$~is~$\mathfrak{P}^{\prime}$\nobreakdash-iso\-lat\-ed in~these groups\textup{)}.

\textup{2.}\hspace{1ex}If~$H$ is~normal in~$A$ and~$B$ and\textup{,}~for~each $p \in \mathfrak{P}$\textup{,} at~least one of~the~conditions~$(\alpha)$\nobreakdash---$(\gamma)$ of~Theorem~\textup{\ref{et02}} is~satisfied\textup{,} then $G$ is~residually an~$\mathcal{FN}_{\mathfrak{P}}$\nobreakdash-group \textup{(}an~$\mathcal{FN}_{\mathfrak{P}} \kern1pt{\cdot}\kern1pt \mathcal{FN}_{\mathfrak{P}}$\nobreakdash-group\textup{)} if and~only~if~$H$ is~$q^{\prime}$\nobreakdash-iso\-lat\-ed in~$A$ and~$B$ for~some $q \in \mathfrak{P}$ \textup{(}respectively\textup{,} $H$~is~$\mathfrak{P}^{\prime}$\nobreakdash-iso\-lat\-ed in~these groups\textup{)}.

\textup{3.}\hspace{1ex}If~$H$ is~periodic\textup{,} then the~following statements are~equivalent\textup{:}

\makebox[1ex]{}\phantom{1.}\makebox[4ex][l]{$(\alpha)$}$G$~is~residually an~$\mathcal{FN}_{\mathfrak{P}}$\nobreakdash-group\textup{;}

\makebox[1ex]{}\phantom{1.}\makebox[4ex][l]{$(\beta\kern.3pt)$}for~some $q \in \mathfrak{P}$\textup{,} $H$~is~$q^{\prime}$\nobreakdash-iso\-lat\-ed in~$A$ and~$B$\textup{,} and\textup{,}~for~any $p \in \mathfrak{P}$\textup{,} there exist sequences of~subgroups described in~Statement~\textup{1} of~Theorem~\textup{\ref{et03};}

\makebox[1ex]{}\phantom{1.}\makebox[4ex][l]{$(\gamma\kern1pt)$}for~some $q \in \mathfrak{P}$\textup{,} $H$~is~$q^{\prime}$\nobreakdash-iso\-lat\-ed in~$A$ and~$B$\textup{,} and\textup{,}~for~any $p \in \mathfrak{P}$\textup{,} there exist normal series described in~Statement~\textup{1} of~Theorem~\textup{\ref{et04}}.
\end{ecorollary}

\pagebreak

When the~groups~$A$ and~$B$ are~finitely generated, the~group $G = \langle A * B;\ H\rangle$ is~residually nilpotent if and~only~if~it is~residually a~finite nilpotent group. Therefore, Statement~1 of~Corollary~\ref{ec03} generalizes

--\hspace{1ex}the~criteria for~the~residual nilpotence of~generalized free products of~a)~two finitely generated abelian groups~\cite[Theorem~5]{Ivanova2004PhD}; b)~two finitely generated nilpotent groups with~a~cyclic amalgamated subgroup~\cite[Theorem~7]{Ivanova2004PhD}; c)~two finitely generated tor\-sion-free nilpotent groups with~an~amalgamated subgroup lying in~the~center of~each free factor~\cite[Corollary of~Theorem~6]{Ivanova2004PhD};

--\hspace{1ex}the~criteria for~the~residual $p$\nobreakdash-fi\-nite\-ness of~generalized free products of~a)~two finitely generated abelian groups~\cite[Corollary~4.8]{KimLeeMcCarron2008KM}; b)~two finitely generated nilpotent groups with~a~cyclic amalgamated subgroup~\cite[Theorem~6]{Ivanova2004PhD} (for~the~case of~tor\-sion-free factors, this result was obtained earlier in~\cite[Theorem~4.4]{KimTang1998JA}); c)~two finitely generated nilpotent groups with~an~amalgamated subgroup lying in~the~center of~each free factor~\cite[Corollary~4.7]{KimLeeMcCarron2008KM}.

\section{Some auxiliary statements}\label{es03}

Everywhere below, if~$\mathcal{C}$ is~an~arbitrary class of~groups and~$X$ is~a~group, then $\mathcal{C}^{*}(X)$ denotes the~family of~normal subgroups of~$X$ defined as~follows: $Y \in \mathcal{C}^{*}(X)$ if and~only~if $X/Y \in \mathcal{C}$.

\begin{eproposition}\label{ep31}
Suppose that $\mathcal{C}$ is~a~class of~groups closed under~taking subgroups and~direct products of~a~finite number of~factors. Suppose also that $X$ is~a~group\textup{,} $Y$~and~$Z$ are~subgroups of~$X$. Then the~following statements hold.

\textup{1.}\hspace{1ex}If~$Y \in \mathcal{C}^{*}(X)$\textup{,} then $Y \cap Z \in \mathcal{C}^{*}(Z)$.

\textup{2.}\hspace{1ex}If~$Y,Z \in \mathcal{C}^{*}(X)$\textup{,} then $Y \cap Z \in \mathcal{C}^{*}(X)$.

\textup{3.}\hspace{1ex}If~$X$ is~residually a~$\mathcal{C}$\nobreakdash-group\textup{,} then\textup{,} for~any finite set $S \subseteq X \setminus \{1\}$\textup{,} there exists a~subgroup $T \in \mathcal{C}^{*}(X)$ such that $T \cap S = \varnothing$.
\end{eproposition}

\begin{proof}
Indeed, if~$Y \in \mathcal{C}^{*}(X)$, then $Z/Y \cap Z \cong ZY/Y \leqslant X/Y \in \mathcal{C}$ and~since $\mathcal{C}$ is~closed under~taking subgroups, $Z/Y \cap Z \in \mathcal{C}$. If~$Y,Z \in \mathcal{C}^{*}(X)$, then, by~Remak's theorem (see, for~example,~\cite[Theorem~4.3.9]{KargapolovMerzlyakov1982}), the~quotient group~$X/Y \cap Z$ can be~embedded into~the~direct product of~the~groups~$X/Y$ and~$X/Z$. Hence, this group belongs to~$\mathcal{C}$ because the~latter is~closed under~taking subgroups and~direct products. If~$X$ is~residually a~$\mathcal{C}$\nobreakdash-group and~$S \subseteq X \setminus \{1\}$, then, for~each element $s \in S$, there exists a~subgroup $T_{s} \in\nolinebreak \mathcal{C}^{*}(X)$ such that $s \notin T_{s}$. Since $S$ is~finite, it~follows from~Statement~2 that the~subgroup $T = \bigcap_{s \in S} T_{s}$ is~desired.
\end{proof}

\begin{eproposition}\label{ep32}
\textup{\cite[Proposition~5]{SokolovTumanova2016SMJ}}
Suppose that $\mathcal{C}$ is~a~class of~groups consisting only of~periodic groups\textup{,} $X$~is~a~group\textup{,} and~$Y$ is~a~subgroup of~$X$. Suppose also that $\mathfrak{P}(\mathcal{C})$ is~the~set of~primes defined as~follows\textup{:} $p \in \mathfrak{P}(\mathcal{C})$ if and~only~if~$p$ divides the~order of~some element of~a~$\mathcal{C}$\nobreakdash-group. If~the~subgroup~$Y$ is~$\mathcal{C}$\nobreakdash-sep\-a\-ra\-ble in~$X$\textup{,} then it is~$\mathfrak{P}(\mathcal{C})^{\prime}$\nobreakdash-iso\-lat\-ed in~this group.
\end{eproposition}

\begin{eproposition}\label{ep33}
\textup{\cite[Proposition~3]{Tumanova2015IVM}}
Suppose that $\mathcal{C}$ is~a~class of~groups closed under~taking quotient groups\textup{,} $X$~is~a~group\textup{,} and~$Y$ is~a~normal subgroup of~$X$. The~quotient group~$X/Y$ is~residually a~$\mathcal{C}$\nobreakdash-group if and~only~if~$Y$ is~$\mathcal{C}$\nobreakdash-sep\-a\-ra\-ble in~$X$.
\end{eproposition}

\begin{eproposition}\label{ep34}
\textup{\cite[Proposition~4]{Tumanova2015IVM}}
Suppose that $\mathcal{C}$ is~a~class of~groups closed under~taking subgroups\textup{,} quotient groups\textup{,} and~direct products of~a~finite number of~factors. Suppose also that $X$ is~a~group and~$Y$ is~a~finite normal subgroup of~$X$. If~$X$ is~residually a~$\mathcal{C}$\nobreakdash-group\textup{,} then $\operatorname{Aut}_{X}(Y) \in \mathcal{C}$.
\end{eproposition}

\pagebreak

\begin{eproposition}\label{ep35}
\textup{\cite[Proposition~18]{Tumanova2015IVM}}
Suppose that $\mathcal{C}$ is~a~class of~tor\-sion-free groups closed under~taking subgroups and~direct products of~a~finite number of~factors. Suppose also that $X$ is~a~group and~$Y$ is~a~subgroup of~$X$. If~$X$ is~residually a~$\mathcal{C}$\nobreakdash-group and~$Y$ is of~finite Hirsch--Zaitsev rank\textup{,} then there exists a~subgroup $Z \in \mathcal{C}^{*}(X)$ such that $Y \cap Z = 1$.
\end{eproposition}

For~the~convenience of~the~reader, we give the~next proposition along with~its proof because the~paper containing it is~actually unavailable.

\begin{eproposition}\label{ep37}
\textup{\cite[Lemma~2.4]{Yakushev2000PIvSU}}
Suppose that $p$ is~a~prime number and~$X$ is~a~finite $p$\nobreakdash-group. Suppose also that $\alpha$ is~an~automorphism of~$X$ and~$\delta$ is~the~automorphism of~the~quotient group~$X/X^{p}X^{\prime}$ induced by~$\alpha$. Then the~order of~$\alpha$ is~a~$p$\nobreakdash-num\-ber if~the~order of~$\delta$ has the~same property.
\end{eproposition}

\begin{proof}
Let $\Gamma_{i}$, $q$, and~$c$ denote the~$i$\nobreakdash-th member of~the~lower central series, the~order, and~the~nilpotency class of~$X$, respectively. Let also $\overline{\alpha}$ be~the~automorphism of~the~quotient group~$X/X^{\prime}$ induced by~$\alpha$. Using induction on~$c$, we first show that if~the~order of~$\overline{\alpha}$ is~a~$p$\nobreakdash-num\-ber, then the~order of~$\alpha$ is~also a~$p$\nobreakdash-num\-ber. Since this is~obvious for~$c = 1$, we can further assume that $c > 1$.

Let $Y$, $\beta$, and~$\bar\beta$ stand for~the~quotient group~$X/\Gamma_{c}$ and~the~automorphisms of~the~groups $Y$ and~$Y/Y^{\prime} = (X/\Gamma_{c})/(X^{\prime}/\Gamma_{c})$ induced by~$\alpha$ and~$\beta$, respectively. It~is~easy to~see that if~$\sigma\colon Y/Y^{\prime} \to X/X^{\prime}$ is~the~isomorphism defined by~the~rule $((x\Gamma_{c})Y^{\prime})\sigma = xX^{\prime}$, then $\bar\beta = \sigma\overline{\alpha}\sigma^{-1}$. The~last relation means that the~order of~$\bar\beta$ is~equal to~the~order of~$\overline{\alpha}$ and~therefore is~a~$p$\nobreakdash-num\-ber. By~the~inductive hypothesis applied to~$Y$, the~order~$r$ of~$\beta$ is~also a~$p$\nobreakdash-num\-ber. It~follows from~the~definition of~$\beta$ that if~$x \in X$ and~$y \in \Gamma_{c-1}$, then $x\alpha^{r} = xw_{1}$ and~$y\alpha^{r} = yw_{2}$ for~suitable elements $w_{1}, w_{2} \in \Gamma_{c}$. Since $\Gamma_{c}$ lies in~the~center of~$X$, we have $[y,x]\alpha^{r} = [y,x][w_{2}, w_{1}] = [y,x]$. Thus, the~automorphism~$\alpha^{r}$ acts identically on~$\Gamma_{c}$ and~therefore $x\alpha^{rq}_{\vphantom{1}} = xw_{1}^{q} = x$ (recall that $q$ denotes the~order of~$X$). Since $x$ is~chosen arbitrarily, it~follows that $\alpha^{rq} = 1$, as~required.

Now\kern-.5pt{} it\kern-.5pt{} remains\kern-.5pt{} to\kern-.5pt{}~prove\kern-.5pt{} that\kern-.5pt{} if\kern-.5pt{}~the\kern-.5pt{}~order\kern-.5pt{} of\kern-.5pt{}~the\kern-.5pt{}~automorphism\kern-.5pt{}~$\delta$\kern-.5pt{} defined\kern-.5pt{} above\kern-.5pt{} is\kern-.5pt{}~a\kern-.5pt{}~$p$\nobreakdash-num\-ber, then the~order of~$\overline{\alpha}$ has the~same property. This is~obvious if~the~quotient group $Z = X/X^{\prime}$ is~trivial. Thus, we can assume further that $Z \ne 1$.

Let $\overline{\gamma}$ be~the~automorphism of~the~quotient group~$Z/Z^{p}$ induced by~the~automorphism $\gamma = \overline{\alpha}$. As~above, it~is~easy to~show that the~order of~$\delta$ is~equal to~the~order of~$\overline{\gamma}$. Let us fix a~decomposition of~$Z$ into~the~direct product of~non-triv\-i\-al cyclic groups with~generators~$z_{1}$,~$z_{2}$,~\ldots,~$z_{n}$. Since $Z \ne 1$, the~relation $n \geqslant 1$ holds, and~the~automorphism~$\gamma$ can be~given by~the~integer matrix~$\Theta = \{\theta_{ij}^{\vphantom{n}}\}_{i,j=1}^{n}$ defined by~the~equalities $z_{i}^{\vphantom{\theta}}\gamma = z_{1}^{\theta_{i1}}z_{2}^{\theta_{i2}}\ldots z_{n}^{\theta_{in}}$, where $1 \leqslant i \leqslant n$. It~is~clear that the~group~$Z/Z^{p}$, the~set $\{z_{1},\, z_{2},\, \ldots,\, z_{n}\}$, and~the~mapping~$\overline{\gamma}$ can be~viewed as~a~linear space over~the~field~$\mathbb{Z}_{p}$, a~basis, and~a~linear operator of~this space, respectively. If~the~elements of~$\Theta$ are~considered representatives of~residue classes modulo~$p$, then the~matrix of~$\overline{\gamma}$ in~the~indicated basis coincides with~$\Theta$ and~therefore $\Theta^{s} = 1$ for~some $p$\nobreakdash-num\-ber~$s$. It~easily follows that if~the~elements of~$\Theta$ are~considered representatives of~residue classes modulo~$p^{k}$ for~some $k \geqslant 1$, then $\Theta^{sp^{k-1}} = 1$. Since the~orders of~the~elements~$z_{1}$, $z_{2}$,~\ldots,~$z_{n}$ divide the~order~$q$ of~$X$, the~last equality means that $\gamma^{sq} = 1$, as~required.
\end{proof}

\begin{eproposition}\label{ep38}
The~following statements hold.

\textup{1.}\hspace{1ex}Any free group is~residually a~finitely generated tor\-sion-free nilpotent group~\textup{\cite{Magnus1935MA, Magnus1937JRAM}}.

\textup{2.}\hspace{1ex}Any polycyclic group is~residually finite~\textup{\cite{Hirsch1946PLMS}}.

\textup{3.}\hspace{1ex}Any finitely generated tor\-sion-free nilpotent group is~residually an~$\mathcal{F}_{p}$\nobreakdash-group for~every prime~$p$~\textup{\cite{Gruenberg1957PLMS}}.

\textup{4.}\hspace{1ex}A~group is~residually an~$\mathcal{FN}_{\mathfrak{P}}$\nobreakdash-group if and~only~if~it is~residually a~$\mathcal{C}$\nobreakdash-group\textup{,} where $\mathcal{C} = \bigcup_{p \in \mathfrak{P}}\mathcal{F}_{p}$.

\textup{5.}\hspace{1ex}Any free group is~residually an~$\mathcal{F}_{p}$\nobreakdash-group for~each prime number~$p$ and~is~residually an~$\mathcal{FN}_{\mathfrak{P}}$\nobreakdash-group for~each non-empty set of~primes~$\mathfrak{P}$.

\textup{6.}\hspace{1ex}If~a~finitely generated group is~residually nilpotent\textup{,} then it is~residually a~$\mathcal{D}$\nobreakdash-group\textup{,} where $\mathcal{D}$ is~the~union of~the~classes~$\mathcal{F}_{p}$ over~all primes~$p$.
\end{eproposition}

\begin{proof}
It is~well known that any finite nilpotent group can be~decomposed into~the~direct product of~its Sylow subgroups (see, for~example, Proposition~\ref{ep41} below). Therefore, every $\mathcal{FN}_{\mathfrak{P}}$\nobreakdash-group is~residually a~$\mathcal{C}$\nobreakdash-group, and~Statement~4 holds. Statement~5 follows from~Statements~1,~3, and~4. Any~nilpotent image of~a~finitely generated group is~a~finitely generated nilpotent group, which is~polycyclic and~therefore residually finite by~Statement~2. In~fact, this image is~residually a~finite nilpotent group since the~class of~nilpotent groups is~closed under~taking subgroups. Thus, Statement~6 follows from~Statement~4.
\end{proof}

\begin{eproposition}\label{ep39}
If~$\mathcal{C}$ is~a~class of~groups consisting only of~finite groups\textup{,} then any extension of~a~free group by~a~$\mathcal{C}$\nobreakdash-group is~residually an~$\mathcal{F}_{p} \kern1pt{\cdot}\kern1pt \mathcal{C}$\nobreakdash-group for~each prime number~$p$.
\end{eproposition}

\begin{proof}
Suppose that $X$ is~an~extension of~a~free group~$Y$ by~a~$\mathcal{C}$\nobreakdash-group, $x \in X \setminus \{1\}$, and~$p$ is~a~prime number. To~prove the~proposition it is~sufficient to~find a~homomorphism of~$X$ onto~a~group from~$\mathcal{F}_{p} \kern1pt{\cdot}\kern1pt \mathcal{C}$ taking $x$ to~a~non-triv\-i\-al element.

If~$x \notin Y$, then the~natural homomorphism $X \to X/Y$ is~the~desired one. Therefore, we can further assume that $x \in Y$. By~Proposition~\ref{ep38}, $Y$~is~residually an~$\mathcal{F}_{p}$\nobreakdash-group. Hence, there exists a~subgroup $M \in \mathcal{F}_{p}^{*}(Y)$ such that $x \notin M$. If~$S$ is~a~set of~representatives for~all cosets of~$Y$ in~$X$ and~$N = \bigcap_{s \in S} s^{-1}Ms$, then $N$ is~a~normal subgroup of~$X$. Since $X/Y \in \mathcal{C}$ and~$\mathcal{C}$ consists of~finite groups, the~set~$S$ is~also finite. It~is~easy to~see that, for~each $s \in S$, the~quotient group~$Y/s^{-1}Ms$ is~isomorphic to~the~$\mathcal{F}_{p}$\nobreakdash-group~$Y/M$. Therefore, $Y/N \in \mathcal{F}_{p}$ by~Proposition~\ref{ep31}. Since $N \leqslant M$ and~$x \notin M$, it~follows that the~natural homomorphism $X \to X/N$ is~desired.
\end{proof}

\section{Some properties of~nilpotent groups}\label{es04}

\begin{eproposition}\label{ep41}
\textup{\cite[Theorem~5.3, Lemma~5.5]{ClementMajewiczZyman2017}}
Suppose that $\mathfrak{P}$ is~a~set of~primes\textup{,} $X$~is~a~locally nilpotent group\textup{,} and~$Y$ is~a~subgroup of~$X$\kern-1pt{}. Then $\mathfrak{P}^{\prime}\textrm{-}\mathfrak{Is}(X,Y) = \mathfrak{P}^{\prime}\textrm{-}\mathfrak{Rt}(X,Y)$ and\textup{,}~if~$Y$ is~normal in~$X$\textup{,} then $\mathfrak{P}^{\prime}\textrm{-}\mathfrak{Is}(X,Y)$ has the~same property. In~particular\textup{,} the~set of~all elements of~finite order of~$X$ is~a~subgroup\textup{,} which can be~decomposed into~the~direct product of~normal \textup{(}in~$X$\textup{)} subgroups~$p\textrm{-}\mathfrak{Rt}(X,1)$ over~all prime numbers~$p$.
\end{eproposition}

Suppose that $\mathcal{C}$ is~a~class of~groups, $X$~is~a~group, and~$Y$ is~a~subgroup of~$X$. Let~us say that $X$ is~\emph{$\mathcal{C}$\nobreakdash-reg\-u\-lar} (\emph{$\mathcal{C}$\nobreakdash-qua\-si-reg\-u\-lar}) with~respect to~$Y$ if,~for~any subgroup $M \in\nolinebreak \mathcal{C}^{*}(Y)$, there exists a~subgroup $N \in \mathcal{C}^{*}(X)$ such that $N \cap Y = M$ (respectively $N \cap Y \leqslant M$). The~next three propositions combine special cases of~Propositions~5.2,~6.3 and~Theorems~2.2,~2.4 from~\cite{Sokolov2023JGT}, which can be~obtained by~replacing the~class~$\mathcal{C}$ that appears in~the~formulations of~these propositions and~theorems~with~$\mathcal{F}_{\mathfrak{P}}$.

\begin{eproposition}\label{ep42}
For~any non-empty set of~primes~$\mathfrak{P}$\textup{,} the~following statements hold.

\textup{1.}\hspace{1ex}The~classes~$\mathcal{BA}_{\mathfrak{P}}$ and~$\mathcal{BN}_{\mathfrak{P}}$ are~closed under~taking subgroups\textup{,} quotient groups\textup{,} and~direct products of~a~finite number of~factors.

\textup{2.}\hspace{1ex}Every abelian $\mathcal{BN}_{\mathfrak{P}}$\nobreakdash-group belongs to~the~class~$\mathcal{BA}_{\mathfrak{P}}$.
\end{eproposition}

\begin{eproposition}\label{ep43}
Suppose that $\mathfrak{P}$ is~a~non-empty set of~primes\textup{,} $X$~is~a~group\textup{,} and~$Y$~is a~subgroup of~$X$. If~there exists a~homomorphism of~$X$ onto~a~$\mathcal{BN}_{\mathfrak{P}}$\nobreakdash-group acting injectively on~$Y$\textup{,} then the~following statements hold.

\textup{1.}\hspace{1ex}The~group~$X$ is~$\mathcal{F}_{\mathfrak{P}}$\nobreakdash-qua\-si-reg\-u\-lar with~respect~to~$Y$.

\textup{2.}\hspace{1ex}If~$Y$ lies in~the~center of~$X$\textup{,} then $X$ is~$\mathcal{F}_{\mathfrak{P}}$\nobreakdash-reg\-u\-lar with~respect~to~$Y$.
\end{eproposition}

\begin{eproposition}\label{ep44}
Suppose that $\mathfrak{P}$ is~a~non-empty set of~primes\textup{,} $X$~is~a~group\textup{,} and~$Y$ is~a~subgroup of~$X$. Suppose also that at~least one of~the~following conditions holds\textup{:}

\makebox[4ex][l]{$(\alpha)$}$X \in \mathcal{BN}_{\mathfrak{P}}$\textup{;}

\makebox[4ex][l]{$(\beta\kern.3pt)$}$X$~is~residually a~$\mathfrak{P}^{\prime}$\nobreakdash-tor\-sion-free $\mathcal{BN}_{\mathfrak{P}}$\nobreakdash-group and~has a~homomorphism onto~such a~group acting injectively~on~$Y$.

Then the~set~$\mathfrak{P}^{\prime}\textrm{-}\mathfrak{Rt}(X,Y)$ is~a~subgroup\textup{,} which is~$\mathcal{FN}_{\mathfrak{P}}$\nobreakdash-sep\-a\-ra\-ble in~$X$. In~particular\textup{,} if~$X$ is~$\mathfrak{P}^{\prime}$\nobreakdash-tor\-sion-free\textup{,} then it is~residually an~$\mathcal{FN}_{\mathfrak{P}}$\nobreakdash-group.
\end{eproposition}

\begin{eproposition}\label{ep45}
If~$\mathfrak{P}$ is~a~finite set of~primes and~$X$ is~a~periodic $\mathfrak{P}^{\prime}$\nobreakdash-tor\-sion-free $\mathcal{BN}_{\mathfrak{P}}$\nobreakdash-group\textup{,} then $X$ is~finite.
\end{eproposition}

\begin{proof}
If~$X$ is~abelian, then it belongs to~the~class~$\mathcal{BA}_{\mathfrak{P}}$ by~Proposition~\ref{ep42}. In~accordance with~the~definition of~this class, a~primary component of~$X$ is~finite if~it corresponds to~a~number from~$\mathfrak{P}$. It~follows that $X$ is~also finite because it is~$\mathfrak{P}^{\prime}$\nobreakdash-tor\-sion-free and~$\mathfrak{P}$ is~finite. In~the~general case, $X$~has a~finite central series with~$\mathcal{BA}_{\mathfrak{P}}$\nobreakdash-fac\-tors, which are~finite as~proven above. Therefore, $X$ is~finite again.
\end{proof}

\section{Generalized free products. Theorem~\ref{et06}}\label{es05}

Recall that \emph{the~generalized free product of~groups~$A$ and~$B$ with~subgroups $H \leqslant A$ and~$K \leqslant B$ amalgamated under~an~isomorphism $\varphi\colon H \to K$} is~the~group~$G$ defined as~follows:

--\hspace{1ex}the~generators of~$G$ are~the~generators of~$A$ and~$B$;

--\hspace{1ex}the~defining relations of~$G$ are~the~relations of~$A$ and~$B$ together with~all possible relations of~the~form $h = h\varphi$, where $h$ and~$h\varphi$ are~some words in~the~generators of~$A$ and~$B$ that define an~element of~$H$ and~its image under~$\varphi$.

It is~well known that the~free factors~$A$ and~$B$ can be~embedded in~$G$ via~the~identical mappings of~their generators. This fact allows~us to~consider~$A$ and~$B$ subgroups of~$G$. Under~this assumption, the~subgroups~$H$ and~$K$ turn~out to~be~equal. Thus, we can use the~notation $G = \langle A * B;\ H\rangle$ and~say that $G$ is~\emph{the~generalized free product of~the~groups~$A$ and~$B$ with~the~subgroup~$H$ amalgamated}.

Let $G = \langle A * B;\ H\rangle$. Recall that the~representation of~an~element $g \in G$ as~a~product $g_{1}g_{2}\ldots g_{n}$ is~said to~be~a~\emph{reduced form} of~this element if $n \geqslant 1$, $g_{1}, g_{2}, \ldots, g_{n} \in A \cup B$, and,~for~$n > 1$, no~two adjacent factors of~the~product (called \emph{syllables} of~the~reduced form) belong simultaneously to~$A$ or~$B$. By~the~normal form theorem for~generalized free products (see, for~example,~\cite[Chapter~IV, Theorem~2.6]{LyndonSchupp1980}), if~$g$ has a~reduced form of~length greater than~$1$, then $g \ne 1$. It~follows that all reduced forms of~$g$ have the~same length, which is~denoted below~by~$\ell(g)$.

If~$R$ and~$S$ are~normal subgroups of~$A$ and~$B$, $R \cap H = S \cap H$, and~$\varphi_{R,S}\colon HR/R \to HS/S$ is~the~mapping taking the~coset~$hR$ to~the~coset~$hS$, then this mapping is~well defined and~is~an~isomorphism of~the~subgroups~$HR/R$ and~$HS/S$. Therefore, we can consider the~generalized free product~$G_{R,S}$ of~the~groups~$A/R$ and~$B/S$ with~the~subgroups~$HR/R$ and~$HS/S$ amalgamated under~the~isomorphism~$\varphi_{R,S}$. It~is~easy to~see that the~natural homomorphisms~$A \to A/R$ and~$B \to B/S$ can be~extended to~a~surjective homomorphism $\rho_{R,S}\colon G \to G_{R,S}$ and~the~kernel of~the~latter is~the~normal closure in~$G$ of~the~set $R \cup S$. We note also that if~$X$ is~a~normal subgroup of~$G$, $R = X \cap A$, and~$S = X \cap B$, then $R$ and~$S$ are~normal subgroups of~the~groups~$A$ and~$B$, respectively, $R \cap H = S \cap H$, and~therefore the~group~$G_{R,S}$ and~the~homomorphism~$\rho_{R,S}$ are~defined.

The~next proposition is~a~special case of~Theorem~5 from~\cite{KarrassSolitar1970TAMS}.

\begin{eproposition}\label{ep51}
Let $G = \langle A * B;\ H\rangle$. If~a~normal subgroup~$N$ of~$G$ intersects~$H$ trivially\textup{,} then it can be~decomposed into~the~\textup{(}ordinary\textup{)} free product of~a~free subgroup and~subgroups\textup{,} each of~which is~conjugate to~$N \cap A$ or~$N \cap B$.
\end{eproposition}

\begin{eproposition}\label{ep52}
Suppose that $G = \langle A * B;\ H\rangle$ and~$\mathcal{C}$ is~a~class of~groups.~If

\medskip

\makebox[4ex][l]{$(\alpha)$}\hfill$\displaystyle\bigcap_{X \in \mathcal{C}^{*}(G)} X \cap A = 1 = \bigcap_{X \in \mathcal{C}^{*}(G)} X \cap B,$\hfill\hspace*{4ex}\hspace*{\parindent}

\smallskip

\makebox[4ex][l]{$(\beta\kern.3pt)$}\hfill$\displaystyle\bigcap_{X \in \mathcal{C}^{*}(G)} H(X \cap A) = H = \bigcap_{X \in \mathcal{C}^{*}(G)} H(X \cap B),$\hfill\hspace*{4ex}\hspace*{\parindent}

\smallskip

\makebox[4ex][l]{$(\gamma\kern1pt)$}\hfill$\displaystyle\forall X,Y \in \mathcal{C}^{*}(G)\ \exists Z \in \mathcal{C}^{*}(G)\ Z \leqslant X \cap Y,$\hfill\hspace*{4ex}\hspace*{\parindent}

\medskip

\noindent
then the~following statements hold.

\textup{1.}\hspace{1ex}For~each element $g \in G \setminus \{1\}$\textup{,} which is~conjugate to~some element of~$A \cup B$\textup{,} there exists a~homomorphism of~$G$ onto~a~group from~$\mathcal{C}$ taking~$g$ to~a~non-triv\-i\-al element.

\textup{2.}\hspace{1ex}For~each element $g \in G \setminus \{1\}$\textup{,} which is~conjugate to~no~element of~$A \cup B$\textup{,} there exists a~homomorphism of~$G$ onto~a~group from~$\Phi \kern1pt{\cdot}\kern1pt \mathcal{C}$ taking~$g$ to~a~non-triv\-i\-al element.

In~particular\textup{,} $G$ is~residually a~$\Phi \kern1pt{\cdot}\kern1pt \mathcal{C}$\nobreakdash-groups \textup{(}recall that $\Phi$ denotes the~class of~all free groups\textup{)}.
\end{eproposition}

\begin{proof}
Let $g \in G \setminus \{1\}$. If~$g$ is~conjugate to~some element $a \in A$, then, by~$(\alpha)$, there exists a~subgroup $X \in \mathcal{C}^{*}(G)$ that does~not contain~$a$. It~is~clear that $g \notin X$, and~therefore the~natural homomorphism~$G \to G/X$ is~the~desired one. A~similar argument can be~used if~$g$ is~conjugate to~an~element of~$B$. Thus, Statement~1 is~proved.

If~$g$ is~conjugate to~no~element of~$A \cup B$ and~$g_{1}g_{2}\ldots g_{n}$ is~its reduced form, then $n > 1$ and~therefore $g_{1}, g_{2}, \ldots, g_{n} \in A \setminus H \cup B \setminus H$. It~follows from~$(\beta)$ that, for~each $i \in \{1,\ldots,n\}$, there exists a~subgroup $X_{i} \in \mathcal{C}^{*}(G)$ such that $g_{i} \notin H(X_{i} \cap A)$ if~$g_{i} \in A \setminus H$, and~$g_{i} \notin H(X_{i} \cap B)$ if~$g_{i} \in B \setminus H$. By~$(\gamma)$, the~intersection $X_{1} \cap X_{2} \cap \ldots \cap X_{n}$ contains a~subgroup $X \in \mathcal{C}^{*}(G)$. Let~$R = X \cap A$ and~$S = X \cap B$. Then $g_{i} \notin HR$ if~$g_{i} \in A \setminus H$, and~$g_{i} \notin HS$ if~$g_{i} \in B \setminus H$, where $1 \leqslant i \leqslant n$. It~follows that $(g_{1}\rho_{R,S})(g_{2}\rho_{R,S})\ldots (g_{n}\rho_{R,S})$ is~a~reduced form of~the~element~$g\rho_{R,S}$ of~length $n>1$ and~therefore $g\rho_{R,S} \ne 1$. Since $\rho_{R,S}$ continues the~natural homomorphisms $A \to A/R$, $B \to B/S$ and~the~kernel of~$\rho_{R,S}$ is~contained in~$X$ as~the~normal closure of~the~set $R \cup S = (X \cap A) \cup (X \cap B)$, we have $G_{R,S}/X\rho_{R,S} \cong G/X \in \mathcal{C}$~and
$$
X\rho_{R,S} \cap A/R = X\rho_{R,S} \cap A\rho_{R,S} = 1 = 
X\rho_{R,S} \cap B\rho_{R,S} = X\rho_{R,S} \cap B/S.
$$
It~follows from~the~last equalities and~Proposition~\ref{ep51} that $X\rho_{R,S}$ is~a~free group and~$\rho_{R,S}$ is~the~desired homomorphism.
\end{proof}

Let us recall~\cite{Sokolov2015CA} that a~class of~groups~$\mathcal{C}$ is~said to~be~a~\emph{root class} if~it contains non-triv\-i\-al groups, is~closed under~taking subgroups, and~satisfies any of~the~following three equivalent conditions:

\makebox[4ex][l]{$(\alpha)$}for~any group~$X$ and~for~any subnormal series $1 \leqslant Z \leqslant Y \leqslant X$, if~$X/Y,\, Y/Z \in \mathcal{C}$, then there exists a~subgroup $T \in \mathcal{C}^{*}(X)$ such that $T \leqslant Z$ (the~\emph{Gruenberg condition});

\makebox[4ex][l]{$(\beta\kern.3pt)$}the~class~$\mathcal{C}$ is~closed under~taking unrestricted wreath products;

\makebox[4ex][l]{$(\gamma\kern1pt)$}the~class~$\mathcal{C}$ is~closed under~taking extensions and,~for~any two groups $X,Y \in \mathcal{C}$, contains the~unrestricted direct product $\prod_{y \in Y} X_{y}$, where $X_{y}$ is~an~isomorphic copy of~$X$ for~each $y \in Y$.

It is~easy to~see that $\mathcal{F}_{\mathfrak{P}}$ is~a~root class for~any non-empty set of~prime numbers~$\mathfrak{P}$. The~classes~$\mathcal{FN}_{\mathfrak{P}}$ (if~$\mathfrak{P}$ contains at~least two prime numbers) and~$\mathcal{BN}_{\mathfrak{P}}$ are~not root classes since they are~not closed under~taking extensions. The~next proposition follows from~Theorems~1,~3 and~Proposition~2 given in~\cite{Sokolov2021SMJ1}.

\begin{eproposition}\label{ep53}
Suppose that $G = \langle A * B;\ H\rangle$ and~$\mathcal{C}$ is~a~root class of~groups. If~$A$ and~$B$ are~residually $\mathcal{C}$\nobreakdash-groups\textup{,} $H$~is~$\mathcal{C}$\nobreakdash-sep\-a\-ra\-ble in~these groups\textup{,} and~$G$ is~$\mathcal{C}$\nobreakdash-qua\-si-reg\-u\-lar with~respect to~$A$ and~$B$\textup{,} then $G$ is~residually a~$\mathcal{C}$\nobreakdash-group.
\end{eproposition}

\begin{eproposition}\label{ep54}
\textup{\cite[Theorem~3]{Tumanova2015IVM}}
Suppose that $G = \langle A * B;\ H\rangle$ and~$\mathcal{C}$ is~a~root class of~groups closed under~taking quotient groups. Suppose also that $H$ is~finite and~normal in~$A$ and~$B$. Then $G$ is~residually a~$\mathcal{C}$\nobreakdash-group if and~only~if~$\operatorname{Aut}_{G}(H) \in \mathcal{C}$.
\end{eproposition}

The~next proposition is~obtained by~combining Proposition~3 and~6 from~\cite{Loginova1999}.

\begin{eproposition}\label{ep55}
Suppose that $G = \langle A * B;\ H\rangle$ and~$p$ is~a~prime. Suppose also that
$$
R = R_{0} \leqslant R_{1} \leqslant \ldots \leqslant R_{n} = A
\quad \text{and} \quad 
S = S_{0} \leqslant S_{1} \leqslant \ldots \leqslant S_{n} = B
$$
are sequences of~subgroups of~$A$ and~$B$ such that

\textup{1)}\hspace{1ex}$R_{i}^{\vphantom{*}} \in \mathcal{F}_{p}^{*}(A)$\textup{,} $S_{i}^{\vphantom{*}} \in \mathcal{F}_{p}^{*}(B)$\textup{,} $0 \leqslant i \leqslant n$\textup{;}

\textup{2)}\hspace{1ex}$R_{i} \cap H = S_{i} \cap H$\textup{,} $0 \leqslant i \leqslant n$\textup{;}

\textup{3)}\hspace{1ex}$|(R_{i+1} \cap H)/(R_{i} \cap H)| \in \{1,p\}$\textup{,} $0 \leqslant i \leqslant n-1$.

\noindent
Then $G_{R,S}$ is~residually an~$\mathcal{F}_{p}$\nobreakdash-group.
\end{eproposition}

\begin{eproposition}\label{ep56}
Suppose that $G = \langle A * B;\ H\rangle$ and~$U$ is~a~subgroup of~$H$\textup{,} which is~normal in~$G$. Suppose also that $\lambda$ and~$\mu$ are~homomorphisms of~the~groups~$A$ and~$B$\textup{,} respectively\textup{,} which map them onto~nilpotent groups and~act injectively on~$H$. Then\textup{,} for~each prime number~$p$\textup{,} the~subgroup $R = p^{\prime}\textrm{-}\mathfrak{Is}(A,\, U \kern1pt{\cdot} \ker\lambda)$ is~normal in~$A$\textup{,} the~subgroup $S = p^{\prime}\textrm{-}\mathfrak{Is}(B,\, U \kern1pt{\cdot} \ker\mu)$ is~normal in~$B$\textup{,} and~$R \cap H = p^{\prime}\textrm{-}\mathfrak{Rt}(H,\,U) = S \cap H$.
\end{eproposition}

\begin{proof}
Since $A\lambda$ and~$B\mu$ are~nilpotent, the~sets~$p^{\prime}\textrm{-}\mathfrak{Rt}(A\lambda, U\lambda)$ and~$p^{\prime}\textrm{-}\mathfrak{Rt}(B\mu, U\mu)$ are~subgroups by~Proposition~\ref{ep41}. It~is~easy to~see that the~sets
$$
p^{\prime}\textrm{-}\mathfrak{Rt}(A,\, U \kern1pt{\cdot} \ker\lambda)
\quad\text{and}\quad
p^{\prime}\textrm{-}\mathfrak{Rt}(B,\, U \kern1pt{\cdot} \ker\mu)
$$
are~the~full pre-images of~these subgroups under~$\lambda$ and~$\mu$, respectively. It~follows that
\begin{align*}
R &= p^{\prime}\textrm{-}\mathfrak{Rt}(A,\, U \kern1pt{\cdot} \ker\lambda),\kern-5ex{} &
\kern-5ex{}R \cap H &= p^{\prime}\textrm{-}\mathfrak{Rt}(H,\, (U \kern1pt{\cdot} \ker\lambda) \cap H), \\
S &= p^{\prime}\textrm{-}\mathfrak{Rt}(B,\, U \kern1pt{\cdot} \ker\mu),\kern-5ex{} &
\kern-5ex{}S \cap H &= p^{\prime}\textrm{-}\mathfrak{Rt}(H,\, (U \kern1pt{\cdot} \ker\mu) \cap H),
\end{align*}
and~the~subgroups~$R$ and~$S$ are~normal in~$A$ and~$B$, respectively, because the~subgroups~$U \kern1pt{\cdot} \ker\lambda$ and~$U \kern1pt{\cdot} \ker\mu$ have the~same property. It~remains to~note that since $\ker\lambda \cap H = 1 = \ker\mu \cap H$, the~equalities $(U \kern1pt{\cdot} \ker\lambda) \cap H = U = (U \kern1pt{\cdot} \ker\mu) \cap H$ hold.
\end{proof}

Everywhere below, if~the~group $G = \langle A * B;\ H\rangle$ and~a~set of~primes~$\mathfrak{P}$ satisfy~$(*)$, then the~phrase ``$\lambda$~and~$\mu$ are~homomorphisms which exist due~to~this condition'' means that $\lambda$ is~a~homomorphism of~$A$, $\mu$~is~a~homomorphism of~$B$, $A\lambda, B\mu \in \mathcal{BN}_{\mathfrak{P}}$, and~$\ker\lambda \cap H = \ker\mu \cap H = 1$.

Suppose that $p$ is~a~prime number and~the~symbols~$U(p)$, $V(p)$, and~$W(p)$ denote the~subgroups~$H^{p}H^{\prime}$, $p^{\prime}\textrm{-}\mathfrak{Is}(A,\, U(p) \kern1pt{\cdot} \ker\lambda)$, and~$p^{\prime}\textrm{-}\mathfrak{Is}(B,\, U(p) \kern1pt{\cdot} \ker\mu)$, respectively. If~$H$ is~normal in~$A$ and~$B$, then $U(p)$ is~normal in~$G$. By~Proposition~\ref{ep56}, it~follows that $V(p)$ and~$W(p)$ are~normal in~$A$ and~$B$, respectively, and~since~$U(p)$ is~obviously $p^{\prime}$\nobreakdash-iso\-lat\-ed in~$H$, the~equalities $V(p) \cap H = U(p) = W(p) \cap H$ hold. Thus, we can consider the~group~$G_{V(p),W(p)}$, the~homomorphism~$\rho_{V(p),W(p)}$, and~also the~group
$$
\operatorname{Aut}_{G_{V(p),W(p)}}(H\rho_{V(p),W(p)}),
$$
which is~correctly defined because $H\rho_{V(p),W(p)}$ is~normal~in~$G_{V(p),W(p)}$.

\begin{etheorem}\label{et06}
Suppose that the~group $G = \langle A * B;\ H\rangle$ and~a~set of~primes~$\mathfrak{P}$ satisfy~$(*)$\textup{,} $\lambda$~and~$\mu$ are~homomorphisms which exist due~to~this condition\textup{,} and~$H$ is~normal in~$A$ and~$B$. Suppose also that\textup{,}~for~each $p \in \mathfrak{P}$\textup{,} the~symbols~$U(p)$\textup{,} $V(p)$\textup{,} and~$W(p)$ denote the~subgroups~$H^{p}H^{\prime}$\textup{,} $p^{\prime}\textrm{-}\mathfrak{Is}(A,\, U(p) \kern1pt{\cdot} \ker\lambda)$\textup{,} and~$p^{\prime}\textrm{-}\mathfrak{Is}(B,\, U(p) \kern1pt{\cdot} \ker\mu)$\textup{,} respectively. If~$\operatorname{Aut}_{G_{V(p),W(p)}}(H\rho_{V(p),W(p)})$ is~a~$p$\nobreakdash-group\textup{,} then Statements~\textup{1\nobreakdash---3} of~Theorem~\textup{\ref{et02}} hold.
\end{etheorem}

\section{Necessary conditions for~the~residual nilpotence. Proof~of~Theorem~\ref{et05}}\label{es06}

\begin{eproposition}\label{ep61}
Suppose that $X$ is~a~group\textup{,} $q$~and~$r$ are~prime numbers\textup{,} and~$x_{1}, x_{2} \in X$. Suppose also that\textup{,} for~any prime~$p$ and~for~any homomorphism~$\sigma$ of~$X$ onto~an~$\mathcal{F}_{p}$\nobreakdash-group\textup{,} if~$p \ne q$\textup{,} then $x_{1}\sigma = 1$\textup{,} and~if~$p \ne r$\textup{,} then $x_{2}\sigma = 1$. If~$q \ne r$ and~$[x_{1}, x_{2}] \ne 1$\textup{,} then $X$ is~not residually a~finite nilpotent group.
\end{eproposition}

\begin{proof}
Assume that $q \ne r$, $[x_{1}, x_{2}] \ne 1$, and~$X$ is~residually a~finite nilpotent group. By~Proposition~\ref{ep38}, $X$~is~residually a~$\mathcal{C}$\nobreakdash-group, where $\mathcal{C}$ is~the~union of~the~classes~$\mathcal{F}_{p}$ over~all primes~$p$. Hence, there exist a~prime number~$p$ and~a~homomorphism~$\sigma$ of~$X$ onto~an~$\mathcal{F}_{p}$\nobreakdash-group taking~$[x_{1}, x_{2}]$ to~a~non-triv\-i\-al element. Since $q \ne r$, at~least one of~the~inequalities $p \ne q$ and~$p \ne r$ holds. Therefore, $x_{1}\sigma = 1$ or~$x_{2}\sigma = 1$, and,~in~both cases, $[x_{1}, x_{2}]\sigma = 1$, contrary to~the~choice~of~$\sigma$.
\end{proof}

\begin{eproposition}\label{ep62}
The~group
$$
G = \big\langle a,b,c;\ a^{9} = [a^{3},b] = [a^{3},c] = c^{-1}bcb = 1\big\rangle
$$
is~not residually a~finite nilpotent group.
\end{eproposition}

\begin{proof}
Let $p$ and~$\sigma$ be~a~prime number and~a~homomorphism of~$G$ onto~an~$\mathcal{F}_{p}$\nobreakdash-group. It~is~clear that if~$p \ne 3$, then $a\sigma = 1$. If~$p \ne 2$, then the~orders~$r$ and~$s$ of~$b\sigma$ and~$c\sigma$ are~odd. Hence, $b\sigma = (c\sigma)^{-s}(b\sigma)(c\sigma)^{s} = (b\sigma)^{(-1)^{s}} = (b\sigma)^{-1}$, $(b\sigma)^{2} = 1 = (b\sigma)^{r}$, and~$b\sigma = 1$. If~$G$ is~considered the~generalized free product of~the~groups~$A = \operatorname{sgp}\{a\}$ and~$B = \operatorname{sgp}\{a^{3},b,c\}$ with~the~subgroup~$H = \operatorname{sgp}\{a^{3}\}$ amalgamated, then $[a,b]$ is~of~length~$4$ and~therefore is~non-triv\-i\-al. By~Proposition~\ref{ep61}, it~follows that $G$ is~not residually a~finite nilpotent group.
\end{proof}

The~next proposition is~given without proof in~\cite{AzarovIvanova1999PIvSU} and~can be~easily verified by~induction~on~$k$.

\begin{eproposition}\label{ep63}
If
\begin{equation}\label{ef01}
\begin{aligned}
\upsilon_{1}(x,y) &= [x,y] = x^{-1}y^{-1}xy,\\
\upsilon_{k+1}(x,y) &= [x,\upsilon_{k}(x,y)] = x^{-1}\upsilon_{k}(x,y)^{-1}x\upsilon_{k}(x,y), \quad k \geqslant 1,
\end{aligned}
\end{equation}
then\textup{,} for~each $k \geqslant 1$\textup{,} the~word~$\upsilon_{k}(x,y)$ in~the~alphabet~$\{x^{\pm 1},y^{\pm 1}\}$ is~of~length~$2^{k+1}$\textup{,} starts with~$x^{-1}y^{-1}$\textup{,} ends with~$xy$\textup{,} and~does~not contain a~subword of~the~form~$x^{\varepsilon}x^{\delta}$ or~$y^{\varepsilon}y^{\delta}$\textup{,} where $\varepsilon, \delta = \pm 1$.
\end{eproposition}

\begin{eproposition}\label{ep64}
\textup{\cite[Lemma~2]{Shirvani1985AM}}
If~a~group~$X$ satisfies a~non-triv\-i\-al identity\textup{,} then it satisfies a~non-triv\-i\-al identity of~the~form
\begin{equation}\label{ef02}
\omega(y, x_{1}, x_{2}) = \omega_{0}(x_{1},x_{2})y^{\varepsilon_{1}}\omega_{1}(x_{1},x_{2})\ldots y^{\varepsilon_{n}}\omega_{n}(x_{1},x_{2}),
\end{equation}
where $n \geqslant 1$\textup{,} $\varepsilon_{1}^{\vphantom{1}}, \ldots, \varepsilon_{n}^{\vphantom{1}} = \pm 1$\textup{,} and~$\omega_{0}^{\vphantom{1}}(x_{1}^{\vphantom{1}},x_{2}^{\vphantom{1}}), \ldots, \omega_{n}^{\vphantom{1}}(x_{1}^{\vphantom{1}},x_{2}^{\vphantom{1}}) \in \{x_{1}^{\pm 1}, x_{2}^{\pm 1}, (x_{1}^{\vphantom{1}}x_{2}^{-1})^{\pm 1}_{\vphantom{1}}\}$.
\end{eproposition}

\begin{proof}[\textup{\textbf{Proof of~Theorem~\ref{et05}}}]
By~the~condition of~the~theorem, there exists a~subgroup~$M$ of~$A$ such that $H$ is~properly contained in~$M$ and~$M$ satisfies a~non-triv\-i\-al identity or~lies in~the~normalizer of~$H$. Let~$N$ denote the~subgroup of~$B$ defined in~the~same way. We prove both statements of~the~theorem by~contradiction. Let us begin with~Statement~2~and~assume that $H$ is~not~$\mathcal{C}$\nobreakdash-sep\-a\-ra\-ble in~$A$, i.e.,~there exists an~element $u \in A \setminus H$ such that $u\theta \in H\theta$ for~any homomorphism~$\theta$ of~$A$ onto~a~$\mathcal{C}$\nobreakdash-group.

If~$H$ is~normal in~$N$, we fix an~element $b \in N \setminus H$ and~put $g_{1}(u,b) = \upsilon_{c}(u, b^{-1}ub)$, where~$c$~is~the~nilpotency class of~$H$ and~$\upsilon_{c}(x,y)$ is~an~element of~the~sequence~\eqref{ef01}. Otherwise, \pagebreak $[N:H] \geqslant 3$ and~$N$ satisfies a~non-triv\-i\-al identity, which can be~considered having the~form~\eqref{ef02}. In~this case, we choose elements $b_{1}, b_{2} \in N \setminus H$ that belong to~different right cosets of~$H$ in~$N$ and~put $g_{2}(u, b_{1}, b_{2}) = \omega(u, b_{1}, b_{2})$.

It follows from~Propositions~\ref{ep63},~\ref{ep64} and~the~relation $b_{1}^{\vphantom{1}}b_{2}^{-1} \in N \setminus H$ that, in~the~generalized free product~$G$, the~elements~$g_{1}(u,b)$ and~$g_{2}(u, b_{1}, b_{2})$ have reduced forms of~length greater than~$1$. Therefore, they are~non-triv\-i\-al. But~if~$\sigma$ is~a~homomorphism of~$G$ onto a~$\mathcal{C}$\nobreakdash-group, then $A\sigma \in \mathcal{C}$ because $\mathcal{C}$ is~closed under~taking subgroups, and~$u\sigma \in H\sigma$ due to~the~choice of~$u$. Hence,

\makebox[3ex][l]{a)}if $H$ is~normal in~$N$, then $g_{1}(u,b)\sigma = \upsilon_{c}(u\sigma, (b\sigma)^{-1}u\sigma b\sigma) = 1$ since $u\sigma \in H\sigma$, $(b\sigma)^{-1}u\sigma b\sigma \in H\sigma$, the~element $\upsilon_{c}(u\sigma, (b\sigma)^{-1}u\sigma b\sigma)$ belongs to~the~$(c+1)$\nobreakdash-th member of~the~lower central series of~$H\sigma$, and~the~nilpotency class of~this group does~not exceed~$c$;

\makebox[3ex][l]{b)}otherwise, $g_{2}(u, b_{1}, b_{2})\sigma = \omega(u\sigma, b_{1}\sigma, b_{2}\sigma) = 1$ since $u\sigma \in N\sigma$, $b_{1}\sigma \in N\sigma$, $b_{2}\sigma \in N\sigma$ and~$N\sigma$ satisfies $\omega(y, x_{1}, x_{2})$.

Thus, we get a~contradiction with~the~fact that $G$ is~residually a~$\mathcal{C}$\nobreakdash-group. The~$\mathcal{C}$\nobreakdash-sep\-a\-ra\-bil\-ity of~$H$ in~$B$ can be~proved in~the~same way.

\smallskip

Now let us turn to~the~proof of~Statement~1 and~assume that, for~any prime number~$p$, $H$~is~not~$p^{\prime}$\nobreakdash-iso\-lat\-ed in~$A$ and~$B$. It~follows that there exist elements $u,v \in A \setminus H \cup B \setminus H$ and~prime numbers~$q$,~$r$ satisfying the~relations $q \ne r$, $u^{q} \in H$, and~$v^{r} \in H$. We~need to~find elements $g_{1}, g_{2} \in G$ such that the~commutator~$[g_{1}, g_{2}]$ has a~reduced form of~length greater than~$1$ (and~therefore is~non-triv\-i\-al), but,~for~any homomorphism~$\sigma$ of~$G$ onto~an~$\mathcal{F}_{p}$\nobreakdash-group, $g_{1}\sigma = 1$ if~$p \ne q$, and~$g_{2}\sigma = 1$ if~$p \ne r$. By~Proposition~\ref{ep61}, the~existence of~such elements means that $G$ is~not residually an~$\mathcal{FN}_{\mathfrak{P}}$\nobreakdash-group, contrary to~the~condition of~the~theorem.

Since $q \ne r$, at~least one of~these numbers does~not~equal~$2$. Without lost of~generality we can assume that $r \geqslant 3$ and~$u \in A \setminus H$. As~above, if~$H$ is~normal in~$M$ (in~$N$), we fix an~element $a \in M \setminus H$ ($b \in N \setminus H$). Otherwise, we choose elements $a_{1}, a_{2} \in M \setminus H$ ($b_{1}, b_{2} \in N \setminus H$) lying in~different right cosets of~$H$ in~$M$ (in~$N$) and~denote by~$\omega_{M}(y, x_{1}, x_{2})$ ($\omega_{N}(y, x_{1}, x_{2})$) the~non-triv\-i\-al identity of~the~form~\eqref{ef02}, which~$M$ (respectively~$N$) satisfies. Below, we define the~elements~$g_{1}$ and~$g_{2}$ in~each of~the~following cases independently.

\smallskip

\makebox[8.25ex][l]{\textit{Case~1.}}$v \in A \setminus H$ and~$H$ is~normal in~$N$.

\smallskip

Since all the~elements~$v$, $v^{2}$,~\ldots,~$v^{r-1}$ lie in~different right cosets of~$H$ in~$A$ and~$r \geqslant 3$, there exists $k \in \{1, 2, \ldots, r-1\}$ such that $v^{k}u^{-1} \in A \setminus H$. We put $g_{1} = \upsilon_{c}(u, b^{-1}ub)$ and~$g_{2} = \upsilon_{c}(v^{k}, b^{-1}v^{k}b)$, where $c$ is~the~nilpotency class of~$H$ and~$\upsilon_{c}(x,y)$ is~an~element of~the~sequence~\eqref{ef01}.

\smallskip

\makebox[8.25ex][l]{\textit{Case~2.}}$v \in A \setminus H$ and~$H$ is~not normal~in~$N$.

\smallskip

Let $g_{1} = \omega_{N}(u, b_{1}, b_{2})$ and~$g_{2} = v^{-1}\omega_{N}(v, b_{1}, b_{2})v$.

\smallskip

\makebox[8.25ex][l]{\textit{Case~3.}}$v \in B \setminus H$ and~$H$ is~normal in~$M$ and~$N$.

\smallskip

As~in~Case~1, there exists $m \in \{1, 2, \ldots, r-1\}$ such that $v^{m}b^{-1} \in B \setminus H$. We put $g_{1} = \upsilon_{c}(b^{-1}ub, u)$ and~$g_{2} = \upsilon_{c}(v^{m}, a^{-1}v^{m}a)$, where $c$ and~$\upsilon_{c}(x,y)$ are~defined as~above.

\smallskip

\makebox[8.25ex][l]{\textit{Case~4.}}$v \in B \setminus H$ and~$H$ is~normal in~$M$, but~not~in~$N$.

\smallskip

Let $g_{1} = u^{-1}\omega_{N}(u, b_{1}, b_{2})u$ and~$g_{2} = \upsilon_{c}(v, a^{-1}va)$.

\smallskip

\makebox[8.25ex][l]{\textit{Case~5.}}$v \in B \setminus H$ and~$H$ is~normal in~$N$, but~not~in~$M$.

\smallskip

Let $g_{1} = \upsilon_{c}(u, b^{-1}ub)$ and~$g_{2} = v^{-1}\omega_{M}(v, a_{1}, a_{2})v$.

\smallskip

\makebox[8.25ex][l]{\textit{Case~6.}}$v \in B \setminus H$ and~$H$ is~not normal both in~$M$ and~in~$N$.

\smallskip

Let $g_{1} = \omega_{N}(u, b_{1}, b_{2})$ and~$g_{2} = \omega_{M}(v, a_{1}, a_{2})$.

\smallskip

It~follows from~the~relations $v^{k}_{\vphantom{1}}u^{-1}_{\vphantom{1}} \in A \setminus H$, $v^{m}_{\vphantom{1}}b^{-1}_{\vphantom{1}} \in B \setminus H$, $a_{1}^{\vphantom{1}}a_{2}^{-1} \in M \setminus H$, $b_{1}^{\vphantom{1}}b_{2}^{-1} \in N \setminus H$ and~Propositions~\ref{ep63},~\ref{ep64} that

--\hspace{1ex}in~each of~Cases~1\nobreakdash---6, $g_{1}$~and~$g_{2}$ have reduced forms of~length greater than~$1$;

--\hspace{1ex}in~Cases~2 and~6, there are~no~cancellations at~the~boundaries of~the~multiplied elements~$g_{1}^{-1}$, $g_{2}^{-1}$, $g_{1}^{\vphantom{1}}$,~$g_{2}^{\vphantom{1}}$, and~therefore $\ell(g_{1}^{-1}g_{2}^{-1}g_{1}^{\vphantom{1}}g_{2}^{\vphantom{1}}) = 2(\ell(g_{1}^{\vphantom{1}})+\ell(g_{2}^{\vphantom{1}})) > 1$;

--\hspace{1ex}in~Cases~1 and~3, the~boundary syllables of~the~words~$g_{2}^{-1}$ and~$g_{1}^{\vphantom{1}}$ are~combined into~one, whence $\ell(g_{1}^{-1}g_{2}^{-1}g_{1}^{\vphantom{1}}g_{2}^{\vphantom{1}}) = 2(\ell(g_{1}^{\vphantom{1}})+\ell(g_{2}^{\vphantom{1}})) - 1 > 1$;

--\hspace{1ex}in~Case~4, $\ell(g_{2}^{-1}g_{1}^{\vphantom{1}}g_{2}^{\vphantom{1}}) = \ell(g_{1}^{\vphantom{1}})+2\ell(g_{2}^{\vphantom{1}})$ and~$\ell(g_{1}^{-1}g_{2}^{-1}g_{1}^{\vphantom{1}}g_{2}^{\vphantom{1}}) \geqslant 2\ell(g_{2}^{\vphantom{1}}) > 1$;

--\hspace{1ex}in~Case~5, $\ell(g_{1}^{-1}g_{2}^{-1}g_{1}^{\vphantom{1}}) = 2\ell(g_{1}^{\vphantom{1}})+\ell(g_{2}^{\vphantom{1}})$ and~$\ell(g_{1}^{-1}g_{2}^{-1}g_{1}^{\vphantom{1}}g_{2}^{\vphantom{1}}) \geqslant 2\ell(g_{1}^{\vphantom{1}}) > 1$.

It~is~easy to~see that, given a~prime~$p$ and~a~homomorphism~$\sigma$ of~$G$ onto~an~$\mathcal{F}_{p}$\nobreakdash-group, we have $u\sigma \in H\sigma$ if~$p \ne q$, and~$v\sigma \in H\sigma$ if~$p \ne r$. It~follows from~this fact and~from~the~definitions of~$c$, $\upsilon_{c}$, $\omega_{M}$, and~$\omega_{N}$ that, in~each of~the~considered cases, $g_{1}\sigma = 1$ if~$p \ne q$, and~$g_{2}\sigma = 1$ if~$p \ne r$, as~required.

Thus, $H$ is~$p^{\prime}$\nobreakdash-iso\-lat\-ed in~$A$ and~$B$ for~some prime number~$p$. By~Statement~2 of~this theorem, $H$~is~also $\mathfrak{P}^{\prime}$\nobreakdash-iso\-lat\-ed in~these groups. If~$p \notin \mathfrak{P}$, it~follows that $H$ is~isolated and,~hence, $q^{\prime}$\nobreakdash-iso\-lat\-ed in~$A$ and~$B$ for~any $q \in \mathfrak{P}$. Thus, Statement~1 is~completely proved.
\end{proof}

\begin{eproposition}\label{ep65}
If~$G = \langle A * B;\ H\rangle$\textup{,} $A \ne H \ne B$\textup{,} and~$H$ is~periodic\textup{,} then Statement~\textup{1} of~Theorem~\textup{\ref{et05}} holds.
\end{eproposition}

\begin{proof}
Let us assume that $G$ is~residually an~$\mathcal{FN}_{\mathfrak{P}}$\nobreakdash-group, but~there exist elements $u,v \in A \setminus\nolinebreak H \cup B\kern-1pt{} \setminus\kern-1pt{} H$ and~prime numbers~$q$,~$r$ such that $q \ne r$ and~$u^{q}, v^{r} \in H$. The~orders of~$u$ and~$v$ can be~decomposed into~the~products $s \kern1pt{\cdot}\kern1pt s^{\prime}$ and~$t \kern1pt{\cdot}\kern1pt t^{\prime}$, where $s$ is~a~$q$\nobreakdash-num\-ber, $s^{\prime}$~is~a~$q^{\prime}$\nobreakdash-num\-ber, $t$~is~an~$r$\nobreakdash-num\-ber, and~$t^{\prime}$ is~an~$r^{\prime}$\nobreakdash-num\-ber. Since $q$ and~$s^{\prime}$ are~co-prime, the~relation $u^{s^{\prime}} \in H$ is~equivalent to~the~inclusion $u \in H$, which contradicts the~choice of~$u$. Hence, $u^{s^{\prime}} \notin H$ and~similarly $v^{t^{\prime}} \notin H$. Using the~condition $A \ne H \ne B$, we can choose some elements $a \in A \setminus H$, $b \in B \setminus H$ and~put $x_{1} = u^{s^{\prime}}$,
$$
x_{2} = \begin{cases}
v^{t^{\prime}} & \text{if}\ u \in A \setminus H\ \text{and}\ v \in B \setminus H\ \text{or}\ v \in A \setminus H\ \text{and}\ u \in B \setminus H;\\
b^{-1}v^{t^{\prime}}b\ & \text{if}\ u,v \in A \setminus H;\\
a^{-1}v^{t^{\prime}}a\ & \text{if}\ u,v \in B \setminus H.
\end{cases}
$$
It~is~easy to~see that $\ell([x_{1}, x_{2}]) = 4$ if~$x_{2} = v^{t^{\prime}}$, and~$\ell([x_{1}, x_{2}]) = 8$ otherwise. Thus, $[x_{1}, x_{2}] \ne 1$. At~the~same time, it~follows from~the~equalities $x_{1}^{s} = 1 = x_{2}^{t}$ and~the~definitions of~$s$ and~$t$ that if~$p$ and~$\sigma$ are~a~prime number and~a~homomorphism of~$G$ onto~an~$\mathcal{F}_{p}$\nobreakdash-group, then $x_{1}\sigma = 1$ whenever $p \ne q$, and~$x_{2}\sigma = 1$ whenever $p \ne r$. Therefore, by~Proposition~\ref{ep61}, $G$~is~not residually a~finite nilpotent group, contrary to~the~assumption.

Thus, we prove that if~$G$ is~residually an~$\mathcal{FN}_{\mathfrak{P}}$\nobreakdash-group, then $H$ is~$p^{\prime}$\nobreakdash-iso\-lat\-ed in~$A$ and~$B$ for~some prime number~$p$. In~addition, if~$G$ is~residually an~$\mathcal{FN}_{\mathfrak{P}}$\nobreakdash-group, then it is~$\mathfrak{P}^{\prime}$\nobreakdash-tor\-sion-free and~therefore $H$ is~$\mathfrak{P}^{\prime}$\nobreakdash-iso\-lat\-ed in~the~free factors. Hence, the~inclusion $p \in \mathfrak{P}$ can be~proved in~the~same way as~in~the~proof of~Theorem~\ref{et05}.
\end{proof}

\section{Conditions for~the~quasi-regularity}\label{es07}

\begin{eproposition}\label{ep71}
Suppose that $G = \langle A * B;\ H\rangle$ and~$\mathcal{C}$ is~a~root class of~groups. If~$G$ is~$\mathcal{C}$\nobreakdash-qua\-si-reg\-u\-lar with~respect to~$H$\textup{,} then it is~$\mathcal{C}$\nobreakdash-qua\-si-reg\-u\-lar with~respect to~$A$ and~$B$.
\end{eproposition}

\begin{proof}
Let $M$ be~a~subgroup from~$\mathcal{C}^{*}(A)$. Since $\mathcal{C}$ is~closed under~taking subgroups, it~follows from~Proposition~\ref{ep31} that $M \cap H \in \mathcal{C}^{*}(H)$. Due~to~the~$\mathcal{C}$\nobreakdash-qua\-si-reg\-u\-lar\-i\-ty of~$G$ with~respect to~$H$, there exists a~subgroup $U \in \mathcal{C}^{*}(G)$ satisfying the~relation $U \cap H \leqslant M \cap H$. Let us put $Q = U \cap A$, $R = U \cap A \cap M$, and~$S = U \cap B$. Then $Q,R \in \mathcal{C}^{*}(A)$ and~$S \in \mathcal{C}^{*}(B)$ by~Proposition~\ref{ep31}.

The~inclusion $U \cap H \leqslant M \cap H$ implies that $R \cap H = U \cap H = S \cap H$. Hence, the~group~$G_{R,S}$ and~the~homomorphism~$\rho_{R,S}$ are~defined. Recall that the~latter continues the~natural homomorphisms~$A \to A/R$, $B \to B/S$ \pagebreak and~the~subgroup~$\ker\rho_{R,S}$ coincides with~the~normal closure of~$R\kern1pt{} \cup\kern1pt{} S$ in~$G$. It~follows that $\ker\rho_{R,S} \leqslant U$\kern-2.5pt{}, $G_{R,S}/U\kern-2pt{}\rho_{R,S} \cong G/U\kern-1.5pt{} \in\nolinebreak \mathcal{C}$\kern-1pt{}, $Q\rho_{R,S} \cong QR/R$, $U\rho_{R,S} \cap A\rho_{R,S} = Q\rho_{R,S}$, and~$U\rho_{R,S} \cap\nolinebreak B\rho_{R,S} = 1$. By~Proposition~\ref{ep51}, the~relations $U\rho_{R,S} \cap H\rho_{R,S} \leqslant U\rho_{R,S} \cap B\rho_{R,S} = 1$ mean that $U\rho_{R,S}$ can be~decomposed into~the~(ordinary) free product of~a~free subgroup~$F$ and~subgroups~$Q_{i}$, $i \in \mathcal{I}$, each of~which is~conjugate to~$Q\rho_{R,S}$. Let us denote by~$\sigma$ the~homomorphism of~$U\rho_{R,S}$ onto~$Q\rho_{R,S}$ which continues the~isomorphisms $Q_{i} \to Q\rho_{R,S}$, $i \in \mathcal{I}$, and~takes the~generators of~$F$~to~$1$.

Since $Q\rho_{R,S} \cong QR/R \leqslant A/R \in \mathcal{C}$ and~$\mathcal{C}$ is~closed under~taking subgroups, we~have $U\kern-1pt{}\rho_{R,S}\sigma \in \mathcal{C}$. Thus, the~class~$\mathcal{C}$ contains the~factors~$G_{R,S}/U\kern-1pt{}\rho_{R,S}$ and~$U\kern-1pt{}\rho_{R,S}/\kern-1pt{}\ker\sigma$ of~the~subnormal series $1 \leqslant \ker\sigma \leqslant U\rho_{R,S} \leqslant G_{R,S}$. By~the~Gruenberg condition (see the~definition of~root class in~Section~\ref{es05}), this implies the~existence of~a~subgroup $N_{R,S} \in \mathcal{C}^{*}(G_{R,S})$ such that $N_{R,S} \leqslant \ker\sigma$. Since $\sigma$ acts injectively on~the~subgroups~$Q_{i}$, $i \in \mathcal{I}$, which are~conjugate to~$Q\rho_{R,S}$, we have $N_{R,S} \cap Q\rho_{R,S} = 1$. It~follows from~this equality and~the~relations $U\rho_{R,S} \cap A\rho_{R,S} = Q\rho_{R,S}$ and~$N_{R,S} \leqslant U\rho_{R,S}$ that $N_{R,S} \cap A\rho_{R,S} = 1$.

Let~$N$ denote the~pre-image of~the~subgroup~$N_{R,S}$ under~the~homomorphism~$\rho_{R,S}$. Then $N \in \mathcal{C}^{*}(G)$, and~since $N_{R,S} \cap\nolinebreak A\rho_{R,S} =\nolinebreak 1$, we have $N \cap A = R \leqslant M$. Therefore, $G$ is~$\mathcal{C}$\nobreakdash-qua\-si-reg\-u\-lar with~respect to~$A$. The~$\mathcal{C}$\nobreakdash-quasi-reg\-u\-lar\-i\-ty of~$G$ with~respect to~$B$ can be~proved in~the~same way.
\end{proof}

\begin{eproposition}\label{ep72}
If~$G = \langle A * B;\ H\rangle$\textup{,} $p$~is~a~prime number\textup{,} $A$~is~$\mathcal{F}_{p}$\nobreakdash-qua\-si-reg\-u\-lar with~respect to~$H$\textup{,} and~$B$ is~$\mathcal{F}_{p}$\nobreakdash-reg\-u\-lar with~respect to~$H$\textup{,} then $G$ is~$\mathcal{F}_{p}$\nobreakdash-qua\-si-reg\-u\-lar with~respect~to~$H$.
\end{eproposition}

\begin{proof}
Let $M$ be~a~subgroup from~$\mathcal{F}_{p}^{*}(H)$. Since $A$ is~$\mathcal{F}_{p}^{\vphantom{*}}$\nobreakdash-qua\-si-reg\-u\-lar with~respect to~$H$, there exists a~subgroup $R \in \mathcal{F}_{p}^{*}(A)$ satisfying the~relation $R \cap H \leqslant M$. It~is~well known that every finite $p$\nobreakdash-group has a~normal series with~factors of~order~$p$. This implies the~existence of~a~sequence of~subgroups
$$
R = R_{0}^{\vphantom{*}} \leqslant R_{1}^{\vphantom{*}} \leqslant \ldots \leqslant R_{n}^{\vphantom{*}} = A
$$
such that $R_{i}^{\vphantom{*}} \in \mathcal{F}_{p}^{*}(A)$, $0 \leqslant i \leqslant n$, and~$|R_{i+1}^{\vphantom{*}}/R_{i}^{\vphantom{*}}| = p$, $0 \leqslant i \leqslant n-1$.

By~Proposition~\ref{ep31}, $R_{i}^{\vphantom{*}} \cap H \in \mathcal{F}_{p}^{*}(H)$, $0 \leqslant i \leqslant n$. Since $B$ is~$\mathcal{F}_{p}^{\vphantom{*}}$\nobreakdash-reg\-u\-lar with~respect to~$H$, for~each $i \in \{0,\ldots,n\}$, there exists a~subgroup $T_{i}^{\vphantom{*}} \in \mathcal{F}_{p}^{*}(B)$ satisfying the~equality $T_{i}^{\vphantom{*}} \cap H = R_{i}^{\vphantom{*}} \cap H$. Without lost of~generality, we can assume that $T_{n}^{\vphantom{*}} = B$. If~$S_{i}^{\vphantom{n}} = \bigcap_{j=i}^{n} T_{j}^{\vphantom{n}}$, $0 \leqslant i \leqslant n$, then $S_{i}^{\vphantom{*}} \in \mathcal{F}_{p}^{*}(B)$ by~Proposition~\ref{ep31}, $S_{n}^{\vphantom{*}} = B$, $S_{i}^{\vphantom{*}} \leqslant S_{i+1}^{\vphantom{*}}$, $0 \leqslant i \leqslant n-1$,~and{\parfillskip=0pt
$$
S_{i}^{\vphantom{n}} \cap H = \bigcap_{j=i}^{n} T_{j}^{\vphantom{n}} \cap H = \bigcap_{j=i}^{n} R_{j}^{\vphantom{n}} \cap H = R_{i}^{\vphantom{n}} \cap H.
$$

}\noindent
In~addition,
\begin{gather*}
R_{i+1}^{\vphantom{*}} \cap H/R_{i}^{\vphantom{*}} \cap H \cong (R_{i+1}^{\vphantom{*}} \cap H)R_{i}^{\vphantom{*}}/R_{i}^{\vphantom{*}} \leqslant R_{i+1}^{\vphantom{*}}/R_{i}^{\vphantom{*}},\quad
|R_{i+1}^{\vphantom{*}}/R_{i}^{\vphantom{*}}| = p,\\
0 \leqslant i \leqslant n-1.
\end{gather*}
Therefore, all the~conditions of~Proposition~\ref{ep55} hold for~the~sequences of~subgroups
$$
R = R_{0} \leqslant R_{1} \leqslant \ldots \leqslant R_{n} = A
\quad\text{and}\quad
S = S_{0} \leqslant S_{1} \leqslant \ldots \leqslant S_{n} = B,
$$
and~thus $G_{R,S}$ is~residually an~$\mathcal{F}_{p}$\nobreakdash-group.

Since $\rho_{R,S}$ continues the~natural homomorphism~$A \to A/R$ and~$A/R \in \mathcal{F}_{p}$, the~subgroup~$H\rho_{R,S}$ is~finite. By~Proposition~\ref{ep31}, this implies the~existence of~a~subgroup $N_{R,S}^{\vphantom{*}} \in \mathcal{F}_{p}^{*}(G_{R,S}^{\vphantom{*}})$ satisfying the~equality $N_{R,S} \cap H\rho_{R,S} = 1$. If~$N$ denotes the~pre-image of~$N_{R,S}$ under~$\rho_{R,S}$, then $N \in \mathcal{F}_{p}^{*}(G)$ and~$N \cap H = R \cap H \leqslant M$. Thus, $G$ is~$\mathcal{F}_{p}$\nobreakdash-qua\-si-reg\-u\-lar with~respect~to~$H$.
\end{proof}

\begin{eproposition}\label{ep73}
If~$p$ is~a~prime number\textup{,} $X$~is~residually an~$\mathcal{F}_{p}$\nobreakdash-group\textup{,} and~$Y$ is~a~locally cyclic subgroup of~$X$\textup{,} then $X$ is~$\mathcal{F}_{p}$\nobreakdash-reg\-u\-lar with~respect~to~$Y$.
\end{eproposition}

\begin{proof}
First of~all, let us note that if~$N \in \mathcal{F}_{p}^{*}(Y)$ and~$k = [Y:N]$, then $N = Y^{k}$. Indeed, it~is~obvious that $Y^{k} \leqslant N$. At~the~same time, since the~exponent of~the~locally cyclic group~$Y/Y^{k}$ divides~$k$, this group is~cyclic and~has an~order at~most~$k$. Hence, $Y^{k} = N$.{\parfillskip=0pt\par}

Suppose now that $M$ is~a~subgroup from~$\mathcal{F}_{p}^{*}(Y)$, $yM$~is~a~generator of~the~finite cyclic subgroup~$Y/M$, $q$~is~the~order of~the~latter, and~$S = \{y,\, y^{2},\, \ldots,\, y^{q-1}\}$. Then $1 \notin S$ and~since $X$ is~residually an~$\mathcal{F}_{p}$\nobreakdash-group, Proposition~\ref{ep31} implies the~existence of~a~subgroup $Z \in \mathcal{F}_{p}^{*}(X)$ satisfying the~relation $S \cap Z = \varnothing$. It is clear that any two non-equal elements of~$S$ lie in~different cosets of~$Z$. Therefore, $q \leqslant r$, where $r = |YZ/Z| = [Y : Z \cap Y]$. Since $q$ and~$r$ are~$p$\nobreakdash-num\-bers, $M = Y^{q}$, and~$Z \cap Y = Y^{r}$, it~follows that $q$ divides~$r$, $Z \cap Y \leqslant M$, $(YZ/Z)/(MZ/Z) \cong Y/M(Y \cap Z) = Y/M$, and~$[YZ/Z : MZ/Z] = q$. If
$$
1 = Z_{0}/Z \leqslant Z_{1}/Z \leqslant \ldots \leqslant Z_{n}/Z = X/Z
$$
is~a~normal series of~the~$\mathcal{F}_{p}$\nobreakdash-group~$X/Z$ with~factors of~order~$p$, then the~factors of~the~series
$$
1 = YZ/Z \cap Z_{0}/Z \leqslant YZ/Z \cap Z_{1}/Z \leqslant \ldots \leqslant YZ/Z \cap Z_{n}/Z = YZ/Z
$$
are~of~order~$1$ or~$p$. Therefore,
$$
\big\{[YZ/Z : YZ/Z \cap Z_{i}/Z] \mid 0 \leqslant i \leqslant n\big\} = \big\{1, p, p^{2}, \ldots, r\big\}.
$$
Since the~finite cyclic group~$YZ/Z$ contains only one subgroup of~index~$q$, we have $Z_{i}/Z \cap YZ/Z = MZ/Z$ for~some $i \in \{0, 1, \ldots, n\}$. It~easily follows from~the~last equality and~the~inclusion $Z \cap Y \leqslant M$ that $Z_{i}^{\vphantom{*}} \cap Y = M$. Since $Z_{i}^{\vphantom{*}}/Z \in \mathcal{F}_{p}^{*}(X/Z)$, the~relation $Z_{i}^{\vphantom{*}} \in \mathcal{F}_{p}^{*}(X)$ holds. Thus, $X$ is~$\mathcal{F}_{p}^{\vphantom{*}}$\nobreakdash-reg\-u\-lar with~respect~to~$Y$.
\end{proof}

\begin{eproposition}\label{ep74}
Suppose that $X$ is~a~group\textup{,} $Y$~is~a~subgroup of~$X$\textup{,} $p$~is~a~prime number\textup{,} and~$\sigma$ is~a~homomorphism of~$X$ such that $\ker\sigma \cap Y \leqslant p^{\prime}\textrm{-}\mathfrak{Is}(Y,1)$. If~$X\sigma$ is~$\mathcal{F}_{p}$\nobreakdash-qua\-si-reg\-u\-lar with~respect to~$Y\sigma$\textup{,} then $X$ is~$\mathcal{F}_{p}$\nobreakdash-qua\-si-reg\-u\-lar with~respect~to~$Y$.
\end{eproposition}

\begin{proof}
Let $M$ be~a~subgroup from~$\mathcal{F}_{p}^{*}(Y)$. Then $M\sigma \in \mathcal{F}_{p}^{*}(Y\sigma)$, and~because $X\sigma$ is~$\mathcal{F}_{p}$\nobreakdash-qua\-si-reg\-u\-lar with~respect to~$Y\sigma$, there exists a~subgroup $N_{\sigma}^{\vphantom{*}} \in \mathcal{F}_{p}^{*}(X\sigma)$ satisfying the~relation $N_{\sigma} \cap Y\sigma \leqslant M\sigma$. It~is~clear that the~pre-image~$N$ of~$N_{\sigma}$ under~$\sigma$ belongs to~$\mathcal{F}_{p}^{*}(X)$. Since $M$ is, obviously, $p^{\prime}$\nobreakdash-iso\-lat\-ed in~$Y$, the~inclusion $p^{\prime}\textrm{-}\mathfrak{Is}(Y,1) \leqslant M$ holds. Now it follows from~the~relations $N_{\sigma} \cap Y\sigma \leqslant M\sigma$, $M \leqslant Y$, and~$\ker\sigma \cap Y \leqslant p^{\prime}\textrm{-}\mathfrak{Is}(Y,1)$ that $N \cap Y \leqslant M(\ker\sigma \cap Y) = M$. Thus, $X$ is~$\mathcal{F}_{p}$\nobreakdash-qua\-si-reg\-u\-lar with~respect~to~$Y$.
\end{proof}

\begin{eproposition}\label{ep75}
Suppose that the~group $G = \langle A * B;\ H\rangle$ and~a~set of~primes~$\mathfrak{P}$ satisfy~$(*)$. Suppose also that $\lambda$~and~$\mu$ are~homomorphisms which exist due~to~this condition\textup{,} $p \in \mathfrak{P}$\textup{,} and~the~symbols~$H(p)$\textup{,} $U(p)$\textup{,} $V(p)$\textup{,} and~$W(p)$ denote the~subgroups~$p\textrm{-}\mathfrak{Is}(H,1)$\textup{,} $H^{p}H^{\prime}$\textup{,} $p^{\prime}\textrm{-}\mathfrak{Is}(A,\, U(p) \kern1pt{\cdot} \ker\lambda)$\textup{,} and~$p^{\prime}\textrm{-}\mathfrak{Is}(B,\, U(p) \kern1pt{\cdot} \ker\mu)$\textup{,} respectively. Then $G$ is~$\mathcal{F}_{p}$\nobreakdash-quasi-reg\-u\-lar with~respect to~$H$ if~at~least one of~the~following statements holds\textup{:}

\makebox[4ex][l]{$(\alpha)$}$H$~is~locally cyclic\textup{;}

\makebox[4ex][l]{$(\beta\kern.3pt)$}$H$~lies in~the~center of~$A$~or~$B$\textup{;}

\makebox[4ex][l]{$(\gamma\kern1pt)$}$H$~is~a~retract of~$A$~or~$B$\textup{;}

\makebox[4ex][l]{$(\kern.95pt\delta\kern1pt)$}$H$~is~periodic\textup{,} and~there exist sequences of~subgroups
\begin{gather*}
1 = Q_{0} \leqslant Q_{1} \leqslant \ldots \leqslant Q_{n} = H(p),\\
R_{0} \leqslant R_{1} \leqslant \ldots \leqslant R_{n} = A, \quad 
S_{0} \leqslant S_{1} \leqslant \ldots \leqslant S_{n} = B
\end{gather*}
such that
\begin{gather*}
R_{i}^{\vphantom{*}} \in \mathcal{F}_{p}^{*}(A),\quad 
S_{i}^{\vphantom{*}} \in \mathcal{F}_{p}^{*}(B),\quad 
R_{i}^{\vphantom{*}} \cap H(p) = Q_{i}^{\vphantom{*}} = S_{i}^{\vphantom{*}} \cap H(p),\quad 
0 \leqslant i \leqslant n,\\
|Q_{i+1}/Q_{i}| = p,\quad 0 \leqslant i \leqslant n-1;
\end{gather*}

\makebox[4ex][l]{$(\kern1pt\varepsilon\kern1pt)$}$H$~is~normal in~$A$ and~$B$\textup{,} and~the~group~$\operatorname{Aut}_{G_{V(p),W(p)}}(H\rho_{V(p),W(p)})$\textup{,} which is~defined due~to~Proposition~\textup{\ref{ep56},} is~a~$p$\nobreakdash-group.
\end{eproposition}

\begin{proof}
First of~all, let us note that, by~Condition~$(*)$, $H$~can be~embedded in~a~$\mathcal{BN\kern-1pt{}}_{\mathfrak{P}}$\nobreakdash-group and~therefore is~itself a~$\mathcal{BN\kern-1pt{}}_{\mathfrak{P}}$\nobreakdash-group due~to~Proposition~\ref{ep42}.\kern-2pt{} We give the~proof of~the~$\mathcal{F}_{p}$\nobreakdash-qua\-si-reg\-u\-lar\-ity of~$G$ independently for~each of~the~statements~$(\alpha)$\nobreakdash---$(\kern1pt\varepsilon\kern1pt)$.

\smallskip

\makebox[4ex][l]{$(\alpha)$}By~Proposition~\ref{ep56}, the~subgroups~$R = p^{\prime}\textrm{-}\mathfrak{Is}(A,\ker\lambda)$ and~$S = p^{\prime}\textrm{-}\mathfrak{Is}(B,\ker\mu)$ are~normal in~$A$ and~$B$, respectively, and~$R \cap H = p^{\prime}\textrm{-}\mathfrak{Rt}(H,1) = S \cap H$. The~groups $A\rho_{R,S} \cong A/R$ and~$B\rho_{R,S} \cong B/S$ have no~$p^{\prime}$\nobreakdash-tor\-sion and~are~isomorphic to~quotient groups of~the~$\mathcal{BN}_{\mathfrak{P}}$\nobreakdash-groups~$A\lambda$ and~$B\mu$. Hence, they belong to~$\mathcal{BN}_{\mathfrak{P}}$ due~to~Proposition~\ref{ep42}. Since $\mathcal{BN}_{\mathfrak{P}} \subseteq \mathcal{BN}_{p}$, these groups are~residually $\mathcal{F}_{p}$\nobreakdash-groups by~Proposition~\ref{ep44} and~therefore they are~$\mathcal{F}_{p}$\nobreakdash-reg\-u\-lar with~respect to~the~locally cyclic subgroup~$H\rho_{R,S}$ by~Proposition~\ref{ep73}. It~now follows from~Proposition~\ref{ep72} that $G_{R,S}$ is~$\mathcal{F}_{p}$\nobreakdash-qua\-si-reg\-u\-lar with~respect to~$H\rho_{R,S}$. Since $\rho_{R,S}$ continues the~natural homomorphism~$A \to A/R$, we have
$$
\ker\rho_{R,S} \cap H = R \cap H = p^{\prime}\textrm{-}\mathfrak{Rt}(H,1) = p^{\prime}\textrm{-}\mathfrak{Is}(H,1).
$$
Hence, $G$ is~$\mathcal{F}_{p}$\nobreakdash-qua\-si-reg\-u\-lar with~respect to~$H$ by~Proposition~\ref{ep74}.

\smallskip

\makebox[4ex][l]{$(\beta\kern.3pt)$}For~definiteness, let $H$ be~central in~$B$. If~$R = \ker\lambda$ and~$S = \ker\mu$, then $R \cap H = S \cap H = 1$ and~therefore the~group~$G_{R,S}$ and~the~homomorphism~$\rho_{R,S}$ are~defined. Since $\mathcal{BN}_{\mathfrak{P}}\kern-1pt{} \subseteq\kern-1pt{} \mathcal{BN}_{p}$ and~$\rho_{R,S}$ continues~$\lambda$ and~$\mu$, we have $A\rho_{R,S}\kern-1pt{} \cong\kern-1pt{} A\lambda\kern-1pt{} \in\kern-1pt{} \mathcal{BN}_{p}$, $B\rho_{R,S}\kern-1pt{} \cong\kern-1pt{} B\mu\kern-1pt{} \in\kern-1pt{} \mathcal{BN}_{p}$, and~$\ker\rho_{R,S} \cap H = 1$. It~follows from~these relations and~Propositions~\ref{ep43},~\ref{ep72}, and~\ref{ep74} that $A\rho_{R,S}$ is~$\mathcal{F}_{p}$\nobreakdash-qua\-si-reg\-u\-lar with~respect to~$H\rho_{R,S}$, $B\rho_{R,S}$~is~$\mathcal{F}_{p}$\nobreakdash-reg\-u\-lar with~respect to~$H\rho_{R,S}$, $G_{R,S}$~is~$\mathcal{F}_{p}$\nobreakdash-qua\-si-reg\-u\-lar with~respect to~$H\rho_{R,S}$, and~$G$ is~$\mathcal{F}_{p}$\nobreakdash-qua\-si-reg\-u\-lar with~respect~to~$H$.

\smallskip

\makebox[4ex][l]{$(\gamma\kern1pt)$}For~definiteness, let $H$ be~a~retract of~$B$, i.e.,~there exists a~normal subgroup~$S$ of~$B$ satisfying the~conditions $S \cap H = 1$ and~$B = SH$. If,~as~above, $R = \ker\lambda$, then $R \cap H = 1 = S \cap H$ and~$G_{R,S} = A\rho_{R,S}$. Since $A\rho_{R,S} \cong A\lambda \in \mathcal{BN}_{\mathfrak{P}} \subseteq \mathcal{BN}_{p}$, it~follows from~these relations and~Propositions~\ref{ep43},~\ref{ep74} that $G_{R,S}$ is~$\mathcal{F}_{p}$\nobreakdash-qua\-si-reg\-u\-lar with~respect to~$H\rho_{R,S}$ and~$G$ is~$\mathcal{F}_{p}$\nobreakdash-qua\-si-reg\-u\-lar with~respect~to~$H$.

\smallskip

\makebox[4ex][l]{$(\kern.95pt\delta\kern1pt)$}Let $M$ be~a~subgroup from~$\mathcal{F}_{p}^{*}(H)$. By~Proposition~\ref{ep41}, the~$\mathcal{BN}_{\mathfrak{P}}$\nobreakdash-group~$H$ can be~decomposed into~the~direct product of~its subgroups~$H(p)$ and~$H(p^{\prime}) = p^{\prime}\textrm{-}\mathfrak{Rt}(H,1)$. It~is~clear that the~subgroups~$M$, $R_{i}$, and~$S_{i}$, $0 \leqslant i \leqslant n$, are~$p^{\prime}$\nobreakdash-iso\-lat\-ed in~$H$, $A$, and~$B$, respectively. Therefore, $H(p^{\prime}) \leqslant M \cap R_{i} \cap S_{i}$, $0 \leqslant i \leqslant n$. Since $R_{i} \cap H(p) = Q_{i} = S_{i} \cap H(p)$, the~equalities $R_{i} \cap H = H(p^{\prime}) \kern1pt{\cdot}\kern1pt Q_{i} = S_{i} \cap H$ hold. We also have
\begin{gather*}
(H(p^{\prime}) \kern1pt{\cdot}\kern1pt Q_{i+1})/(H(p^{\prime}) \kern1pt{\cdot}\kern1pt Q_{i}) \cong Q_{i+1}/Q_{i}(H(p^{\prime}) \cap Q_{i+1}) = Q_{i+1}/Q_{i},
\quad
|Q_{i+1}/Q_{i}| = p,\\
0 \leqslant i \leqslant n-1.
\end{gather*}
Hence, if~$R = R_{0}$ and~$S = S_{0}$, then, by~Proposition~\ref{ep55}, $G_{R,S}$~is~residually an~$\mathcal{F}_{p}$\nobreakdash-group. Since the~group~$A\rho_{R,S} \cong A/R$ is~finite, Proposition~\ref{ep31} implies the~existence of~a~subgroup $N_{R,S}^{\vphantom{*}} \in \mathcal{F}_{p}^{*}(G_{R,S}^{\vphantom{*}})$ satisfying the~equality $N_{R,S}^{\vphantom{*}} \cap A\rho_{R,S}^{\vphantom{*}} = 1$. Let~$N$ denote the~pre-image of~$N_{R,S}^{\vphantom{*}}$ under~$\rho_{R,S}^{\vphantom{*}}$. Then $N \in \mathcal{F}_{p}^{*}(G)$ and $N \cap A = \ker\rho_{R,S}^{\vphantom{*}} \cap A = R$ because $\rho_{R,S}^{\vphantom{*}}$ continues the~natural homomorphism~$A \to\nolinebreak A/R$. It~follows that $N \cap H = R \cap H = H(p^{\prime}) \leqslant M$ and~$G$~is~$\mathcal{F}_{p}$\nobreakdash-qua\-si-reg\-u\-lar with~respect~to~$H$.

\smallskip

\makebox[4ex][l]{$(\kern1pt\varepsilon\kern1pt)$}Let again $M$ be~a~subgroup from~$\mathcal{F}_{p}^{*}(H)$. Suppose also that $r\kern-1pt{} =\kern-1pt{} [H:M]$ and~$L\kern-1pt{} =\nolinebreak\kern-1pt{} H^{r}$\kern-1pt{}. As~above, to~complete the~proof it is~sufficient to~indicate a~subgroup $N \in \mathcal{F}_{p}^{*}(G)$ satisfying the~condition $N \cap H \leqslant M$. If~$r = 1$, then $G$ is~the~required subgroup because $G \cap H = H = M$ and~$G \in \mathcal{F}_{p}^{*}(G)$. Thus, we can assume that $r > 1$ and~therefore $L \leqslant U(p)$. Let us put $R = p^{\prime}\textrm{-}\mathfrak{Is}(A,\, L \kern1pt{\cdot} \ker\lambda)$ and~$S = p^{\prime}\textrm{-}\mathfrak{Is}(B,\, L \kern1pt{\cdot} \ker\mu)$. Since $L$ is~normal in~$G$ and~$p^{\prime}$\nobreakdash-iso\-lat\-ed in~$H$, it~follows from~Proposition~\ref{ep56} that $R$ is~normal in~$A$, $S$~is~normal in~$B$, and~$R \cap H = L = S \cap H$. The~group~$H/L$ is~periodic, has no~$p^{\prime}$\nobreakdash-tor\-sion, and~satisfies the~relations
$$
H/L = H/H \cap R \cong HR/R \cong H\rho_{R,S} \in \mathcal{BN}_{\mathfrak{P}} \subseteq \mathcal{BN}_{p}
$$
because $\rho_{R,S}$ continues the~natural homomorphism~$A\kern-1.5pt{} \to\kern-2pt{} A/\kern-1pt{}R$, $H\kern-2pt{} \in\kern-1pt{} \mathcal{BN}\kern-1pt{}_{\mathfrak{P}}$\kern-.5pt{}, and~the~class~$\mathcal{BN}\kern-1pt{}_{\mathfrak{P}}$ is~closed under~taking quotient groups due~to~Proposition~\ref{ep42}. Hence, this group is~finite by~Proposition~\ref{ep45}. Let us show that $\operatorname{Aut}_{G_{R,S}}(H\rho_{R,S}) \in \mathcal{F}_{p}$.

Since $U(p)$ is~normal in~$G$, the~subgroup~$U_{R,S} = U(p)\rho_{R,S}$ is~normal in~$G_{R,S}$. By~Proposition~\ref{ep37}, if~$\alpha\kern-1.5pt{} \in\kern-3pt{} \operatorname{Aut}_{G\kern-.5pt{}_{R,S}}(\kern-.5pt{}H\kern-1.5pt{}\rho\kern-.5pt{}_{R,S})$\kern-.5pt{}, $\overline{\alpha}$ is~the~automorphism of~the~group $H\kern-1.5pt{}\rho\kern-.5pt{}_{R,S}/(\kern-.5pt{}H\kern-1.5pt{}\rho\kern-.5pt{}_{R,S})^{p}(\kern-.5pt{}H\kern-1.5pt{}\rho\kern-.5pt{}_{R,S})^{\prime}$ induced by~$\alpha$, and~the~order of~$\overline{\alpha}$ is~a~$p$\nobreakdash-num\-ber, then the~order of~$\alpha$ is~also a~$p$\nobreakdash-num\-ber. It~is~easy to~see that $(H\rho_{R,S})^{p}(H\rho_{R,S})^{\prime} = U_{R,S}$ and~if~$\alpha = \widehat{g}\vert_{H\rho_{R,S}}$ for~some $g \in G_{R,S}$, then $\overline{\alpha} = \widehat{gU_{R,S}}\vert_{H\rho_{R,S}/U_{R,S}}$ (here the~symbol~$\widehat{x}$ denotes the~inner automorphism given by~$x$). Thus, it~is~sufficient to~prove that $\operatorname{Aut}_{G_{R,S}/U_{R,S}}(H\rho_{R,S}/U_{R,S})$ is~a~$p$\nobreakdash-group.

Let us note that $U(p)R = V(p)$ and~$U(p)S = W(p)$. Indeed, it~follows from~the~inclusion $L \leqslant U(p)$ that $R \leqslant V(p)$ and~therefore $U(p)R \leqslant V(p)$. Since $R$ is~$p^{\prime}$\nobreakdash-iso\-lat\-ed in~$A$ and,~as~proven above, $HR/R$~is~finite, the~quotient group~$A/R$ is~$p^{\prime}$\nobreakdash-tor\-sion-free and~the~finite subgroup~$U(p)R/R$ is~$p^{\prime}$\nobreakdash-iso\-lat\-ed in~this group. Hence, the~subgroup~$U(p)R$ is~$p^{\prime}$\nobreakdash-iso\-lat\-ed in~$A$. The~last fact and~the~inclusion $U(p) \kern1pt{\cdot} \ker\lambda \leqslant U(p)R$ mean that $V(p) \leqslant U(p)R$. The~equality $U(p)S = W(p)$ can be~proved in~the~same~way.

It is~easy to~see that $G_{R,S}/U_{R,S}$ is~the~generalized free product of~the~groups~$A\rho_{R,S}/U_{R,S}$ and~$B\rho_{R,S}/U_{R,S}$ with~the~subgroup~$H\rho_{R,S}/U_{R,S}$ amalgamated. Since $\rho_{R,S}$ continues the~natural homomorphisms~$A \to A/R$ and~$B \to B/S$, the~following relations hold:
\begin{gather*}
A\rho_{R,S}/U_{R,S} \cong (A/R)/(U(p)R/R) \cong A/V(p),\\
B\rho_{R,S}/U_{R,S} \cong (B/S)/(U(p)S/S) \cong B/W(p).
\end{gather*}
The~indicated isomorphisms define an~isomorphism of~$G_{R,S}/U_{R,S}$ onto~$G_{V(p),W(p)}$, which maps~$H\rho_{R,S}/U_{R,S}$ onto~$H\rho_{V(p),W(p)}$ and~thus induces an~isomorphism of~the~groups
$$
\operatorname{Aut}_{G_{R,S}/U_{R,S}}(H\rho_{R,S}/U_{R,S})
\quad\text{and}\quad
\operatorname{Aut}_{G_{V(p),W(p)}}(H\rho_{V(p),W(p)}).
$$
By~the~condition of~the~proposition, the~latter is~a~$p$\nobreakdash-group, as~required.

So,~$\operatorname{Aut}_{G_{R,S}}(H\rho_{R,S}) \in \mathcal{F}_{p}$ and~$H\rho_{R,S}$ is~finite. Due~to~Propositions~\ref{ep54} and~\ref{ep31}, it~follows that $G_{R,S}$ is~residually an~$\mathcal{F}_{p}$\nobreakdash-group and~there exists a~subgroup $N_{R,S}^{\vphantom{*}} \in \mathcal{F}_{p}^{*}(G_{R,S}^{\vphantom{*}})$ satisfying the~relation $N_{R,S} \cap H\rho_{R,S} = 1$. If~$N$ denotes the~pre-image of~$N_{R,S}$ under~$\rho_{R,S}$, then $N \in \mathcal{F}_{p}^{*}(G)$ and~$N \cap H \leqslant \operatorname{ker}\rho_{R,S} \cap H = R \cap H$ because $\rho_{R,S}$ continues the~natural homomorphism~$A \to A/R$. Since $R \cap H = L$ and~$L \leqslant M$ by~the~definition of~$L$, it~follows that $N$ is~the~desired subgroup.
\end{proof}

\section{Main theorem}\label{es08}

As~noted in~Section~\ref{es02}, the~next theorem serves as~the~main tool that allows one to~use results on~the~residual $p$\nobreakdash-fi\-nite\-ness of~the~group $G = \langle A * B;\ H\rangle$ to~study the~residual nilpotence of~this group.

\begin{etheorem}\label{et07}
Suppose that $G = \langle A * B;\ H\rangle$ and~$\mathfrak{P}$ is~a~non-empty set of~primes. Suppose also that the~subgroups~$1$ and~$H$ are~$\mathcal{FN}_{\mathfrak{P}}$\nobreakdash-sep\-a\-ra\-ble in~each of~the~free factors and\textup{,}~for~any $p \in \mathfrak{P}$\textup{,} $G$~is~$\mathcal{F}_{p}$\nobreakdash-qua\-si-reg\-u\-lar with~respect to~$H$. Then the~following statements hold.

\textup{1.}\hspace{1ex}The~group~$G$ is~residually a~$\Phi \kern1pt{\cdot}\kern1pt \mathcal{FN}_{\mathfrak{P}}$\nobreakdash-group and~residually an~$\mathcal{F}_{p} \kern1pt{\cdot}\kern1pt \mathcal{FN}_{\mathfrak{P}}$\nobreakdash-group for~each prime~$p$.

\textup{2.}\hspace{1ex}If\textup{,}~for~some $q \in \mathfrak{P}$\textup{,} the~set $U = q^{\prime}\textrm{-}\mathfrak{Rt}(H,1)$ is~a~subgroup\textup{,} which is~$\mathcal{F}_{q}$\nobreakdash-sep\-a\-ra\-ble in~$H$\textup{,} and~$H$ is~$\mathcal{F}_{q}$\nobreakdash-sep\-a\-ra\-ble in~each of~the~free factors\textup{,} then $G$ is~residually an~$\mathcal{FN}_{\mathfrak{P}}$\nobreakdash-group.
\end{etheorem}

\begin{proof}
By~Proposition~\ref{ep71}, for~any $p \in \mathfrak{P}$, the~$\mathcal{F}_{p}$\nobreakdash-qua\-si-reg\-u\-lar\-i\-ty of~$G$ with~respect to~$H$ implies the~$\mathcal{F}_{p}$\nobreakdash-qua\-si-reg\-u\-lar\-i\-ty of~this group with~respect to~$A$ and~$B$. Let us show that Conditions~$(\alpha)$\nobreakdash---$(\gamma)$ from~Proposition~\ref{ep52} hold for~the~group~$G$ and~the~class $\mathcal{C} = \mathcal{FN}_{\mathfrak{P}}$.{\parfillskip=0pt{}\par}

Indeed, if~$X,Y \in \mathcal{FN}_{\mathfrak{P}}^{*}(G)$ and~$Z = X \cap Y$, then $Z \in \mathcal{FN}_{\mathfrak{P}}^{*}(G)$ by~Proposition~\ref{ep31}. Hence, $(\gamma)$~holds. Suppose that $K = 1$ or~$K = H$, and~$a \in A \setminus K$. By~the~condition of~the~theorem, $K$~is~$\mathcal{FN}_{\mathfrak{P}}^{\vphantom{*}}$\nobreakdash-sep\-a\-ra\-ble in~$A$ and~therefore $a \notin KM$ for~some subgroup $M \in\nolinebreak \mathcal{FN}_{\mathfrak{P}}^{*}(A)$. Due~to~Proposition~\ref{ep41}, the~$\mathcal{FN}_{\mathfrak{P}}^{\vphantom{*}}$\nobreakdash-group~$A/M$ can be~decomposed into~the~direct product of~its Sylow subgroups~$T_{1}^{\vphantom{n}}/M$, $T_{2}^{\vphantom{n}}/M$,~\ldots,~$T_{n}^{\vphantom{n}}/M$. If~$M_{i}^{\vphantom{n}}$, $1 \leqslant\nolinebreak i \leqslant\nolinebreak n$, denotes the~subgroup~$\prod_{j \ne i} T_{j}^{\vphantom{n}}$, then $\bigcap_{i=1}^{n} M_{i}^{\vphantom{n}}/M = 1$ and~therefore $\bigcap_{i=1}^{n} M_{i}^{\vphantom{n}} =\nolinebreak M$. It~is~clear also that if~the~subgroup~$T_{i}^{\vphantom{n}}/M$ corresponds to~a~number~$p_{i}^{\vphantom{n}} \in \mathfrak{P}$, then $M_{i}^{\vphantom{n}} \in \mathcal{F}_{p_{i}}^{*}(A)$. Hence, it~follows from~the~$\mathcal{F}_{p_{i}}^{\vphantom{*}}$\nobreakdash-qua\-si-reg\-u\-lar\-i\-ty of~$G$ with~respect to~$A$ that there exists a~subgroup $N_{i}^{\vphantom{*}} \in\nolinebreak \mathcal{F}_{p_{i}}^{*}(G)$ satisfying the~relation $N_{i}^{\vphantom{n}} \cap A \leqslant M_{i}^{\vphantom{n}}$. If~$N = \bigcap_{i=1}^{n} N_{i}^{\vphantom{n}}$, then $N \cap A \leqslant\nolinebreak \bigcap_{i=1}^{n} M_{i}^{\vphantom{n}} =\nolinebreak M$, $a \notin K(N \cap A)$, and,~by~Proposition~\ref{ep31}, $N \in \mathcal{FN}_{\mathfrak{P}}^{*}(G)$ because $\mathcal{F}_{p_{i}}^{*}(G) \subseteq \mathcal{FN}_{\mathfrak{P}}^{*}(G)$ for~any $i \in \{1, 2, \ldots, n\}$. Since the~element~$a$ is~chosen arbitrarily, we have
$$
\bigcap_{X \in \mathcal{FN}_{\mathfrak{P}}^{*}(G)} K(X \cap A) = K.
$$
A~similar argument can be~used to~prove the~equalities from~Conditions~$(\alpha)$ and~$(\beta)$ of~Proposition~\ref{ep52}, which relate to~the~group~$B$. Thus, all the~conditions of~the~indicated proposition hold.

By~Propositions~\ref{ep52} and~\ref{ep39}, $G$~is~residually a~$\Phi \kern1pt{\cdot}\kern1pt \mathcal{FN}_{\mathfrak{P}}$\nobreakdash-group and~any $\Phi \kern1pt{\cdot}\kern1pt \mathcal{FN}_{\mathfrak{P}}$\nobreakdash-group is~residually an~$\mathcal{F}_{p} \kern1pt{\cdot}\kern1pt \mathcal{FN}_{\mathfrak{P}}$\nobreakdash-group for~every $p \in \mathfrak{P}$. Hence, Statement~1 of~the~theorem holds. Let us prove Statement~2.

Due~to~Proposition~\ref{ep32}, the~$\mathcal{F}_{q}$\nobreakdash-sep\-a\-ra\-bil\-ity of~$H$ in~$A$ and~$B$ implies that $H$ is~$q^{\prime}$\nobreakdash-iso\-lat\-ed in~these groups. Therefore, $q^{\prime}\textrm{-}\mathfrak{Rt}(A,1) = U = q^{\prime}\textrm{-}\mathfrak{Rt}(B,1)$. It~follows that $U$ is~normal in~$A$, in~$B$, and~hence in~$G$. Let us show that this subgroup is~$\mathcal{F}_{q}$\nobreakdash-sep\-a\-ra\-ble in~$A$ and~$B$, $H/U$~is~$\mathcal{F}_{q}$\nobreakdash-sep\-a\-ra\-ble in~$A/U$ and~$B/U$, and~$G/U$ is~$\mathcal{F}_{q}$\nobreakdash-qua\-si-reg\-u\-lar with~respect to~$A/U$ and~$B/U$.

Indeed, let $a \in A \setminus U$. To~prove the~$\mathcal{F}_{q}^{\vphantom{*}}$\nobreakdash-sep\-a\-ra\-bil\-ity of~$U$ in~$A$, it~is~sufficient to~find a~subgroup $X \in \mathcal{F}_{q}^{*}(A)$ satisfying the~relation $a \notin UX$. Since $H$ is~$\mathcal{F}_{q}^{\vphantom{*}}$\nobreakdash-sep\-a\-ra\-ble in~$A$, if~$a \notin\nolinebreak H$, there exists a~subgroup $L \in \mathcal{F}_{q}^{*}(A)$ such that $a \notin HL$. Then $a \notin UL$ and~therefore $L$ is~the~desired subgroup. Let~$a \in H$. It~follows from~the~$\mathcal{F}_{q}^{\vphantom{*}}$\nobreakdash-sep\-a\-ra\-bil\-ity of~$U$ in~$H$ that $a \notin UM$ for~some subgroup $M \in \mathcal{F}_{q}^{*}(H)$. Since $G$ is~$\mathcal{F}_{q}^{\vphantom{*}}$\nobreakdash-qua\-si-reg\-u\-lar with~respect to~$H$, there exists a~subgroup $N \in \mathcal{F}_{q}^{*}(G)$ satisfying the~condition $N \cap H \leqslant M$. If~$a \in U(N \cap A)$, it follows from~the~inclusions $a \in H$ and~$U \leqslant H$ that $a \in U(N \cap H) \leqslant UM$, in~contradiction with~the~choice of~$M$. Hence, $a \notin U(N \cap A)$, and~because $N \cap A \in \mathcal{F}_{q}^{*}(A)$ by~Proposition~\ref{ep31}, $N \cap A$ is~the~desired subgroup.

The~$\mathcal{F}_{q}^{\vphantom{*}}$\nobreakdash-sep\-a\-ra\-bil\-ity of~$H/U$ in~$A/U$ is~proved similarly. Namely, if~$aU \in (A/U) \setminus (H/U)$, then $a \in A \setminus H$ and~$a \notin HL$ for~some subgroup $L \in \mathcal{F}_{q}^{*}(A)$ since $H$ is~$\mathcal{F}_{q}^{\vphantom{*}}$\nobreakdash-sep\-a\-ra\-ble in~$A$. It~easily follows that $LU/U \in \mathcal{F}_{q}^{*}(A/U)$ and~$aU \notin (H/U)(LU/U)$. Therefore, $H/U$ is~$\mathcal{F}_{q}^{\vphantom{*}}$\nobreakdash-sep\-a\-ra\-ble in~$A/U$.

Finally, if~$M/U \in \mathcal{F}_{q}^{*}(A/U)$, then $M \in \mathcal{F}_{q}^{*}(A)$ and~because $G$ is~$\mathcal{F}_{q}^{\vphantom{*}}$\nobreakdash-qua\-si-reg\-u\-lar with~respect to~$A$, there exists a~subgroup $N \in \mathcal{F}_{q}^{*}(G)$ satisfying the~condition $N \cap A \leqslant M$. It~is~obvious that $N$ is~$q^{\prime}$\nobreakdash-iso\-lat\-ed in~$G$. Hence, $U \leqslant N$ and~$N/U \cap A/U \leqslant M/U$. Since $N/U \in \mathcal{F}_{q}^{*}(G/U)$, it~follows that $G/U$ is~$\mathcal{F}_{q}$\nobreakdash-qua\-si-reg\-u\-lar with~respect to~$A/U$. The~$\mathcal{F}_{q}$\nobreakdash-sep\-a\-ra\-bil\-ity of~$U$ and~$H/U$ in~$B$ and~$B/U$, respectively, as~well as~the~$\mathcal{F}_{q}$\nobreakdash-qua\-si-reg\-u\-lar\-i\-ty of~$G/U$ with~respect to~$B/U$ can be~proved in~a~similar way.

By~Proposition~\ref{ep33}, the~$\mathcal{F}_{q}$\nobreakdash-sep\-a\-ra\-bil\-ity of~$U$ in~$A$ and~$B$ means that $A/U$ and~$B/U$ are~residually $\mathcal{F}_{q}$\nobreakdash-groups. Therefore, it~follows from~Proposition~\ref{ep53} that $G/U$ is~residually an~$\mathcal{F}_{q}$\nobreakdash-group. Let us now choose an~element $g \in G \setminus \{1\}$ and~indicate a~homomorphism of~$G$ onto~an~$\mathcal{FN}_{\mathfrak{P}}$\nobreakdash-group that takes~$g$ to~a~non-triv\-i\-al element.

If~$g \notin U$, then $gU \ne 1$ and~the~natural homomorphism~$G \to G/U$ can be~extended to~the~desired~one because $G/U$ is~residually an~$\mathcal{F}_{q}$\nobreakdash-group and~$\mathcal{F}_{q} \subseteq \mathcal{FN}_{\mathfrak{P}}$. If~$g \in U$, then the~existence of~the~required homomorphism follows from~Statement~1 of~Proposition~\ref{ep52}, which, as~above, is~applied to~the~group~$G$ and~the~class~$\mathcal{C} = \mathcal{FN}_{\mathfrak{P}}$. Thus, $G$ is~residually an~$\mathcal{FN}_{\mathfrak{P}}$\nobreakdash-group.
\end{proof}

\begin{ecorollary}\label{ec04}
Suppose that the~group $G = \langle A * B;\ H\rangle$ and~a~set of~primes~$\mathfrak{P}$ satisfy~$(*)$. Suppose also that\textup{,} for~each $p \in \mathfrak{P}$\textup{,} $G$~is~$\mathcal{F}_{p}$\nobreakdash-qua\-si-reg\-u\-lar with~respect to~$H$. Then Statements~\textup{1} and~\textup{2} of~Theorem~\textup{\ref{et01}} hold.
\end{ecorollary}

\begin{proof}
By~Condition~$(*)$, $A$~and~$B$ are~residually $\mathcal{FN}_{\mathfrak{P}}$\nobreakdash-groups. This is~equivalent to~the $\mathcal{FN}_{\mathfrak{P}}$\nobreakdash-sep\-a\-ra\-bil\-ity of~the~trivial subgroup in~these groups. Thus, Statement~2 of~Theorem~\ref{et01} immediately follows from~Theorem~\ref{et07}. Let us prove Statement~1.

If~$q \in \mathfrak{P}$, then $\mathcal{F}_{q} \subseteq \mathcal{FN}_{\mathfrak{P}}$ and~the~$\mathcal{F}_{q}$\nobreakdash-sep\-a\-ra\-bil\-ity of~$H$ in~$A$ and~$B$ means that $H$ is~$\mathcal{FN}_{\mathfrak{P}}$\nobreakdash-sep\-a\-ra\-ble in~these groups. By~Condition~$(*)$, $H$~can be~embedded in~a~$\mathcal{BN}_{\mathfrak{P}}$\nobreakdash-group. Since $\mathcal{BN}_{\mathfrak{P}} \subseteq \mathcal{BN}_{q}$, Propositions~\ref{ep42},~\ref{ep41}, and~\ref{ep44} imply that $H \in \mathcal{BN}_{q}$, the~set $U = q^{\prime}\textrm{-}\mathfrak{Rt}(H,1)$ is~a~subgroup, and~this subgroup is~$\mathcal{F}_{q}$\nobreakdash-sep\-a\-ra\-ble in~$H$. Thus, Statement~1 of~Theorem~\ref{et01} follows from~Statement~2 of~Theorem~\ref{et07}.
\end{proof}

\section{Proof of~Theorems~\ref{et01}---\ref{et04}~and~\ref{et06}}\label{es09}

\begin{eproposition}\label{ep91}
Suppose that the~group $G = \langle A * B;\ H\rangle$ and~a~set of~primes~$\mathfrak{P}$ satisfy~$(*)$. Then $H$ is~$\mathfrak{P}^{\prime}$\nobreakdash-iso\-lat\-ed in~$A$ and~$B$ if and~only~if~it is~$\mathcal{FN}_{\mathfrak{P}}$\nobreakdash-sep\-a\-ra\-ble in~these groups. In~particular\textup{,} if~$H$ is~periodic\textup{,} then it is~$\mathcal{FN}_{\mathfrak{P}}$\nobreakdash-sep\-a\-ra\-ble in~$A$ and~$B$.
\end{eproposition}

\begin{proof}
Let $\lambda$ and~$\mu$ be~homomorphisms which exist due~to~Condition~$(*)$. By~Proposition~\ref{ep41}, the~torsion subgroup of~the~nilpotent group~$A\lambda$ can be~decomposed into~the~direct product of~the~subgroups~$\mathfrak{P}\textrm{-}\mathfrak{Rt}(A\lambda, 1)$ and~$\mathfrak{P}^{\prime}\textrm{-}\mathfrak{Rt}(A\lambda, 1)$, normal in~$A\lambda$. Since $A$ is~residually an~$\mathcal{FN}_{\mathfrak{P}}$\nobreakdash-group, it~follows from~Proposition~\ref{ep32} that both~$A$ and~$H$ have no~$\mathfrak{P}^{\prime}$\nobreakdash-tor\-sion. The~injectivity of~$\lambda$ on~$H$ implies that $H\lambda \cap \mathfrak{P}^{\prime}\textrm{-}\mathfrak{Rt}(A\lambda, 1) = 1$. Therefore, the~composition of~$\lambda$ and~the~natural homomorphism~$A\lambda \to A\lambda/\mathfrak{P}^{\prime}\textrm{-}\mathfrak{Rt}(A\lambda, 1)$ still acts injectively on~$H$. By~Proposition~\ref{ep42}, the~quotient group~$A\lambda/\mathfrak{P}^{\prime}\textrm{-}\mathfrak{Rt}(A\lambda, 1)$ belongs to~the~class~$\mathcal{BN}_{\mathfrak{P}}^{\kern1pt{}\text{tf}}$ of~all $\mathfrak{P}^{\prime}$\nobreakdash-tor\-sion-free $\mathcal{BN}_{\mathfrak{P}}^{\vphantom{\text{tf}}}$\nobreakdash-groups. Since $\mathcal{FN}_{\mathfrak{P}}^{\vphantom{\text{tf}}} \subseteq \mathcal{BN}_{\mathfrak{P}}^{\kern1pt{}\text{tf}}$, $A$~is~residually a~$\mathcal{BN}_{\mathfrak{P}}^{\kern1pt{\text{tf}}}$\nobreakdash-group. Now it follows from~Proposition~\ref{ep44} that if~$H$ is~$\mathfrak{P}^{\prime}$\nobreakdash-iso\-lat\-ed in~$A$, then it is~$\mathcal{FN}_{\mathfrak{P}}$\nobreakdash-sep\-a\-ra\-ble in~this group. The~converse statement is~implied by~Proposition~\ref{ep32}. Similarly, it~can be~proved that $H$ is~$\mathfrak{P}^{\prime}$\nobreakdash-iso\-lat\-ed in~$B$ if and~only~if~it is~$\mathcal{FN}_{\mathfrak{P}}$\nobreakdash-sep\-a\-ra\-ble in~this group.

Assume now that $H$ is~a~periodic group. If~an~element $a \in A$ and~a~number $q \in \mathfrak{P}^{\prime}$ are~such that $a^{q} \in H$, then the~order~$r$ of~$a$ is~finite. Since $A$ is~$\mathfrak{P}^{\prime}$\nobreakdash-tor\-sion-free, $r$~and~$q$ are~co-prime. Hence, $a \in H$ and~$H$ is~$\mathfrak{P}^{\prime}$\nobreakdash-iso\-lat\-ed in~$A$. As~above, it~follows that $H$ is~$\mathcal{FN}_{\mathfrak{P}}$\nobreakdash-sep\-a\-ra\-ble in~$A$. The~$\mathcal{FN}_{\mathfrak{P}}$\nobreakdash-sep\-a\-ra\-bil\-ity of~$H$ in~$B$ can be~proved similarly.
\end{proof}

\begin{eproposition}\label{ep92}
Suppose that $G = \langle A * B;\ H\rangle$\textup{,} $\mathfrak{P}$ is~a~non-empty set of~primes\textup{,} $H$~is~periodic and~can be~embedded in~a~$\mathcal{BN}_{\mathfrak{P}}$\nobreakdash-group. If~$G$ is~residually an~$\mathcal{F}_{q} \kern1pt{\cdot}\kern1pt \mathcal{FN}_{\mathfrak{P}}$\nobreakdash-group for~each $q \in \mathfrak{P}$\textup{,} then\textup{,} for~any $p \in \mathfrak{P}$\textup{,} there exist sequences of~subgroups described in~Statement~\textup{1} of~Theorem~\textup{\ref{et03}}.
\end{eproposition}

\begin{proof}
Let $p \in \mathfrak{P}$, and~let $H(p) = p\textrm{-}\mathfrak{Is}(H,1)$. Since $H$ can be~embedded in~a~$\mathcal{BN}_{\mathfrak{P}}$\nobreakdash-group, it~follows from~Propositions~\ref{ep41} and~\ref{ep42} that $H(p) = p\textrm{-}\mathfrak{Rt}(H,1)$, $H \in \mathcal{BN}_{\mathfrak{P}}$, and~$H(p) \in \mathcal{BN}_{\mathfrak{P}} \subseteq \mathcal{BN}_{p}$. By~Proposition~\ref{ep45}, the~$p^{\prime}$\nobreakdash-tor\-sion-free periodic $\mathcal{BN}_{p}$\nobreakdash-group~$H(p)$ is~finite. Let us show that there exists a~homomorphism of~$G$ onto~an~$\mathcal{F}_{p}$\nobreakdash-group which acts injectively~on~$H(p)$.

Indeed, if~$\mathfrak{P} = \{p\}$, then $\mathcal{F}_{p} \kern1pt{\cdot}\kern1pt \mathcal{FN}_{\mathfrak{P}} = \mathcal{F}_{p}$, $G$~is~residually an~$\mathcal{F}_{p}$\nobreakdash-group, and~the~existence of~the~required homomorphism follows from~Proposition~\ref{ep31}. Therefore, we can assume further that $\mathfrak{P}$ contains a~number~$q$ which is~not equal to~$p$. Since $G$ is~residually an~$\mathcal{F}_{q} \kern1pt{\cdot}\kern1pt \mathcal{FN}_{\mathfrak{P}}$\nobreakdash-group and~the~class~$\mathcal{F}_{q} \kern1pt{\cdot}\kern1pt \mathcal{FN}_{\mathfrak{P}}$ is~closed under~taking subgroups and~direct products of~a~finite number of~factors, Proposition~\ref{ep31} implies that $G$ has a~homomorphism~$\sigma$ onto~an~$\mathcal{F}_{q} \kern1pt{\cdot}\kern1pt \mathcal{FN}_{\mathfrak{P}}$\nobreakdash-group which acts injectively on~$H(p)$. By~the~definition of~the~class~$\mathcal{F}_{q} \kern1pt{\cdot}\kern1pt \mathcal{FN}_{\mathfrak{P}}$, $G\sigma$~contains an~$\mathcal{F}_{q}$\nobreakdash-sub\-group $M \in \mathcal{FN}_{\mathfrak{P}}^{*}(G\sigma)$. Let $\varepsilon$ denote the~natural homomorphism~$G\sigma \to G\sigma/M$. Due~to~Proposition~\ref{ep41}, the~$\mathcal{FN}_{\mathfrak{P}}$\nobreakdash-group~$G\sigma\varepsilon$ can be decomposed into~the~direct product of~the~subgroups~$p\textrm{-}\mathfrak{Rt}(G\sigma\varepsilon, 1) \in \mathcal{F}_{p}$ and~$p^{\prime}\textrm{-}\mathfrak{Rt}(G\sigma\varepsilon, 1)$. Since $H(p)$ is~a~$p$\nobreakdash-group and~$q \ne p$, the~following relations hold:
$$
H(p)\sigma \cap M = 1 = H(p)\sigma \varepsilon \cap p^{\prime}\textrm{-}\mathfrak{Rt}(G\sigma\varepsilon, 1).
$$
Hence,\kern-1pt{} the\kern-.5pt{}~composition\kern-.5pt{} of\kern-.5pt{}~$\sigma$\kern-2pt{},\kern-1pt{} $\varepsilon$\kern-.5pt{},\kern-1pt{} and\kern-.5pt{}~the\kern-.5pt{}~natural\kern-.5pt{} homomorphism\kern-.5pt{} $G\sigma\varepsilon\kern-1pt{} \to\kern-1pt{} G\sigma \varepsilon/(p^{\prime}\textrm{-}\mathfrak{Rt}(G\sigma\varepsilon\kern-.5pt{},\kern-.5pt{} 1)\kern-.5pt{})$ is~the~desired~map.

So,~there exists a~subgroup $N \in \mathcal{F}_{p}^{*}(G)$ that meets~$H(p)$ trivially. Let
$$
1 = X_{0}/N \leqslant X_{1}/N \leqslant \ldots \leqslant X_{m}/N = G/N
$$
be~a~normal series of~the~$\mathcal{F}_{p}$\nobreakdash-group~$G/N$ with~factors of~order~$p$. Suppose also that the~series
$$
1 = G_{0}/N \leqslant G_{1}/N \leqslant \ldots \leqslant G_{n}/N = G/N
$$
is~obtained from~the~previous one by~removing some terms in~such a~way that there are~no~duplicates among the~subgroups~$G_{i} \cap H(p)$, $0 \leqslant i \leqslant n$. Then
$$
|G_{i+1} \cap H(p)/G_{i} \cap H(p)| = p,\quad 
0 \leqslant i \leqslant n-1,
$$
and,~by~Proposition~\ref{ep31}, the~subgroups~$Q_{i} = G_{i} \cap H(p)$, $R_{i} = G_{i} \cap A$, and~$S_{i} = G_{i} \cap B$, $0 \leqslant i \leqslant n$, are~the~desired ones.
\end{proof}

\begin{proof}[\textup{\textbf{Proof of~Theorems~\ref{et01},~\ref{et03},~and~\ref{et06}}}]
Here, by~Statements~1\nobreakdash---3 of~Theorem~\ref{et06} we mean Statements~1\nobreakdash---3 of~Theorem~\ref{et02}. Theorem~\ref{et05} and~Propositions~\ref{ep91},~\ref{ep65} imply

--\hspace{1ex}Statement~2 of~Theorem~\ref{et06} and~the~necessity of~the~conditions of~Statement~3 of~the same theorem;

--\hspace{1ex}the~$\mathcal{FN}_{\mathfrak{P}}$\nobreakdash-sep\-a\-ra\-bil\-ity of~$H$ in~$A$ and~$B$, when the~conditions of~Theorem~\ref{et03} hold;

--\hspace{1ex}the~property of~$H$ to~be~$q^{\prime}$\nobreakdash-iso\-lat\-ed, which is~contained in~Statement~2 of~Theorem~\ref{et03}.

Since $\mathcal{FN}_{\mathfrak{P}} \subseteq \mathcal{F}_{q} \kern1pt{\cdot}\kern1pt \mathcal{FN}_{\mathfrak{P}}$ for~any $q \in \mathfrak{P}$, Proposition~\ref{ep92} completes the~proof of~Statement~2 of~Theorem~\ref{et03} and~the~necessity of~the~conditions of~Statement~3 of~the~same theorem. Statements~1 and~2 of~Theorem~\ref{et01}, Statements~1 of~Theorems~\ref{et03} and~\ref{et06}, and~also the~sufficiency of~the~conditions of~Statements~3 of~Theorems~\ref{et03} and~\ref{et06} follow from~Proposition~\ref{ep75} and~Corollary~\ref{ec04}.
\end{proof}

\begin{proof}[\textup{\textbf{Proof of~Theorem~\ref{et02}}}]
Let $\lambda$ and~$\mu$ be~homomorphisms which exist due~to~Condition~$(*)$. Suppose also that $p \in \mathfrak{P}$ and~$U(p)$, $V(p)$, and~$W(p)$ are~the~subgroups defined in~the~same way as~in~Theorem~\ref{et06}. Denote by~$\sigma$ the~mapping of~$G/U(p)$ to~$G_{V(p),W(p)}$ which takes a~coset~$gU(p)$ to~the~element~$g\rho_{V(p),W(p)}$, where $g \in G$. It~follows from~the~relations $U(p) \leqslant V(p) \cap W(p) \leqslant \ker\rho_{V(p),W(p)}$ that this mapping is~correctly defined. It~is~also easy to~see that $\sigma$~is~a~surjective homomorphism. Since
$$
(H/U(p))\sigma = H\rho_{V(p),W(p)},\quad
(A/U(p))\sigma = A\rho_{V(p),W(p)},\quad
(B/U(p))\sigma = B\rho_{V(p),W(p)},
$$
$\sigma$~induces a~homomorphism of~the~group $\mathfrak{G}(p) = \operatorname{Aut}_{G/U(p)}(H/U(p))$ onto~the~group
$$
\overline{\mathfrak{G}}(p) = \operatorname{Aut}_{G_{V(p),W(p)}}(H\rho_{V(p),W(p)}),
$$
which maps the~subgroups
$$
\mathfrak{A}(p) = \operatorname{Aut}_{A/U(p)}(H/U(p))
\quad\text{and}\quad
\mathfrak{B}(p) = \operatorname{Aut}_{B/U(p)}(H/U(p))\kern4pt\mbox{}
$$
onto~the~subgroups
$$
\overline{\mathfrak{A}}(p) = \operatorname{Aut}_{A\rho_{V(p),W(p)}}(H\rho_{V(p),W(p)})
\quad\text{and}\quad
\overline{\mathfrak{B}}(p) = \operatorname{Aut}_{B\rho_{V(p),W(p)}}(H\rho_{V(p),W(p)}),
$$
respectively. \pagebreak

Let $\varepsilon$ denote the~natural homomorphism~$A \to A/V(p)$. Since 
$\rho_{V(p),W(p)}$ extends~$\varepsilon$, we have $\overline{\mathfrak{A}}(p) \cong \operatorname{Aut}_{A\varepsilon}(H\varepsilon)$. It~follows from~Proposition~\ref{ep42} and~the~inclusions $A\lambda \in \mathcal{BN}_{\mathfrak{P}}$ and~$\ker\lambda \leqslant V(p)$ that $A\varepsilon \in \mathcal{BN}_{\mathfrak{P}}$. Obviously, $A\varepsilon$ is~a~$p^{\prime}$\nobreakdash-tor\-sion-free group, which belongs to~$\mathcal{BN}_{p}$ because $p \in \mathfrak{P}$. Hence, it~is~residually an~$\mathcal{F}_{p}$\nobreakdash-group by~Proposition~\ref{ep44}, and~$H\varepsilon \in\nolinebreak \mathcal{BN}_{p}$ by~Proposition~\ref{ep42}. Since $U(p)$ is~$p^{\prime}$\nobreakdash-iso\-lat\-ed in~$H$, Proposition~\ref{ep56} implies that $H \cap V(p) = U(p)$ and~$H\varepsilon \cong H/H \cap V(p) = H/U(p)$. Therefore, $H\varepsilon$~is~a~periodic $p^{\prime}$\nobreakdash-tor\-sion-free $\mathcal{BN}_{p}$\nobreakdash-group, which is~finite by~Proposition~\ref{ep45}. Now it follows from~Proposition~\ref{ep34} that $\operatorname{Aut}_{A\varepsilon}(H\varepsilon) \in \mathcal{F}_{p}$. The~inclusion $\overline{\mathfrak{B}}(p) \in \mathcal{F}_{p}$ can be~proved in~a~similar~way.

It~is~clear that the~group~$\overline{\mathfrak{G}}(p)$ is~generated by~its subgroups $\overline{\mathfrak{A}}(p)$ and~$\overline{\mathfrak{B}}(p)$. Moreover, when abelian, it~coincides with~the~product of~these subgroups. Therefore, any of~Conditions~$(\alpha)$\nobreakdash---$(\gamma)$ implies that $\overline{\mathfrak{G}}(p) \in \mathcal{F}_{p}$. Thus, we can use Theorem~\ref{et06}, which says that Statements~1\nobreakdash---3 hold.
\end{proof}

\begin{proof}[\textup{\textbf{Proof of~Theorem~\ref{et04}}}]
The~$\mathcal{FN}_{\mathfrak{P}}$\nobreakdash-sep\-a\-ra\-bil\-ity of~$H$ in~$A$ and~$B$ follows from~Proposition~\ref{ep91}. Let us fix a~number $p \in \mathfrak{P}$ and~put $H(p) = p\textrm{-}\mathfrak{Is}(H,1)$.

It is~clear that the~set $A(p) = p\textrm{-}\mathfrak{Rt}(A,1)$ is~invariant under~any automorphism of~$A$. Since the~subgroup~$\operatorname{sgp}\{\tau A\}$ can be~embedded into~a~$\mathcal{BN}_{\mathfrak{P}}$\nobreakdash-group and,~in~particular, is~nilpotent, Proposition~\ref{ep41} and~the~equality $A(p) = p\textrm{-}\mathfrak{Rt}(\operatorname{sgp}\{\tau A\},1)$ imply that $A(p)$ is~a~subgroup. The~obvious relation $p\textrm{-}\mathfrak{Rt}(H,1) = A(p) \cap H$ means that the~set~$p\textrm{-}\mathfrak{Rt}(H,1)$ is~also a~subgroup and~hence coincides with~$H(p)$. Since $A(p)\kern-1pt{} \leqslant\kern-1.5pt{} \operatorname{sgp}\{\tau A\}$ and~$\mathcal{BN\kern-1pt{}}_{\mathfrak{P}}\kern-1pt{} \subseteq\nolinebreak\kern-1pt{} \mathcal{BN\kern-1pt{}}_{p}$, it~follows from~Propositions~\ref{ep42},~\ref{ep45}, and~\ref{ep43} that $A(p) \in \mathcal{BN}_{p}$, $A(p)$~is~finite, and~$A$ is $\mathcal{F}_{p}$\nobreakdash-qua\-si-reg\-u\-lar with~respect to~$A(p)$. The~last fact and~the~inclusion $1 \in \mathcal{F}_{p}^{*}(A(p))$ imply the~existence of~a~subgroup $C \in \mathcal{F}_{p}^{*}(A)$ satisfying the~relation $C \cap A(p) = 1$. A~similar argument allows one to~assert that $B(p)$ is~a~finite normal subgroup of~$B$ and~$D \cap\nolinebreak B(p) = 1$ for~some subgroup $D \in \mathcal{F}_{p}^{*}(B)$. Since $p$ is~chosen arbitrarily and~Theorem~\ref{et03} is~already proved, it~remains to~show that, for~this~$p$, the~sequences of~subgroups described in~Statement~1 of~Theorem~\ref{et03} and~the~normal series from~Statement~1 of~the~theorem to~be~proved exist simultaneously. Let
\begin{gather*}
1 = Q_{0} \leqslant Q_{1} \leqslant \ldots \leqslant Q_{n} = H(p),\\
R_{0} \leqslant R_{1} \leqslant \ldots \leqslant R_{n} = A,
\quad\text{and}\quad
S_{0} \leqslant S_{1} \leqslant \ldots \leqslant S_{n} = B
\end{gather*}
be~the~sequences from~Theorem~\ref{et03}. It~is~clear that 
\begin{gather*}
H(p) = p\textrm{-}\mathfrak{Rt}(H,1) \leqslant A(p) \cap B(p),\\
\begin{aligned}
(R_{0} \cap C) \cap H(p) \leqslant R_{0} \cap C \cap A(p) &= 1 = R_{0} \cap H(p),\\
(S_{0} \cap D) \cap H(p) = S_{0} \cap D \cap B(p) &= 1 = S_{0} \cap H(p).
\end{aligned}
\end{gather*}
By~Proposition~\ref{ep31}, $R_{0} \cap C \in \mathcal{F}_{p}^{*}(A)$ and~$S_{0} \cap D \in \mathcal{F}_{p}^{*}(B)$. Therefore, we can further assume without loss of~generality that $R_{0} \leqslant C$, $S_{0} \leqslant D$, and~hence $R_{0} \cap A(p) = 1 = S_{0} \cap B(p)$.{\parfillskip=0pt\par}

Since any finite $p$\nobreakdash-group has a~normal series with~factors of~order~$p$, the~sequences
\begin{equation}\label{ef03}
R_{0} \leqslant R_{1} \leqslant \ldots \leqslant R_{n} = A
\quad \text{and} \quad
S_{0} \leqslant S_{1} \leqslant \ldots \leqslant S_{n} = B
\end{equation}
can be~refined in~such a~way that the~orders of~their factors are~also equal to~$p$. Let
$$
R_{0} = M_{0} \leqslant M_{1} \leqslant \ldots \leqslant M_{r} = A
\quad\text{and}\quad
S_{0} = N_{0} \leqslant N_{1} \leqslant \ldots \leqslant N_{s} = B
$$
be~the~results of~this refinement. If~$K_{i} = M_{i} \cap A(p)$ and~$L_{j} = N_{j} \cap B(p)$, where $0 \leqslant i \leqslant r$ and~$0 \leqslant j \leqslant s$, then the~members of~the~series
\begin{equation}\label{ef04}
1 = K_{0} \leqslant K_{1} \leqslant \ldots \leqslant K_{r} = A(p)
\quad \text{and} \quad
1 = L_{0} \leqslant L_{1} \leqslant \ldots \leqslant L_{s} = B(p)
\end{equation}
are normal in~$A$ and~$B$, and~their factors are~of~order~$1$~or~$p$.

The~equality $H(p) = p\textrm{-}\mathfrak{Rt}(H,1)$ means that
\begin{gather*}
A(p) \cap H = H(p) = B(p) \cap H,\quad 
K_{i} \cap H = M_{i} \cap H(p),\quad 
L_{j} \cap H = N_{j} \cap H(p),\\ 
0 \leqslant i \leqslant r,\quad
0 \leqslant j \leqslant s.
\end{gather*}
Since $R_{i} \cap H(p) = Q_{i} = S_{i} \cap H(p)$, $0 \leqslant i \leqslant n$, and~$|Q_{i+1}/Q_{i}| = p$, $0 \leqslant i \leqslant n-1$, the~members of~any refinements of~the~sequences~\eqref{ef03} meet~$H(p)$ in~the~subgroups~$Q_{i}$, $0 \leqslant i \leqslant n$. It~follows that
\begin{gather*}
\{K_{i} \cap H \mid 0 \leqslant i \leqslant r\} = 
\{M_{i} \cap H(p) \mid 0 \leqslant i \leqslant r\} = 
\{Q_{i} \mid 0 \leqslant i \leqslant n\} =\\
\{N_{j} \cap H(p) \mid 0 \leqslant j \leqslant s\} = 
\{L_{j} \cap H \mid 0 \leqslant j \leqslant s\}.
\end{gather*}
Thus, all the~conditions of~Theorem~\ref{et04} are~satisfied for~the~series~\eqref{ef04} if~we remove duplicate members from~them.

Now let
\begin{equation}\label{ef05}
1 = A_{0} \leqslant A_{1} \leqslant \ldots \leqslant A_{k} = A(p) 
\quad \text{and} \quad
1 = B_{0} \leqslant B_{1} \leqslant \ldots \leqslant B_{m} = B(p)
\end{equation}
be~the~series described in~Statement~1 of~the~theorem to~be~proved. Then
$$
1 = A_{0} \cap H(p) \leqslant A_{1} \cap H(p) \leqslant \ldots \leqslant A_{k} \cap H(p) = H(p)
$$
is~a~normal series of~$H(p)$, whose factors are~of~order~$1$ or~$p$. Since
$$
\mbox{}\kern2pt\{A_{i} \cap H \mid 0 \leqslant i \leqslant k\} = 
\{B_{j} \cap H \mid 0 \leqslant j \leqslant m\},
$$
the~equality
$$
\{A_{i} \cap H(p) \mid 0 \leqslant i \leqslant k\} = 
\{B_{j} \cap H(p) \mid 0 \leqslant j \leqslant m\}
$$
holds. Hence, by~removing certain members of~\eqref{ef05}, we can get normal series
$$
1 = U_{0} \leqslant U_{1} \leqslant \ldots \leqslant U_{n} = A(p)
\quad\text{and}\quad
1 = V_{0} \leqslant V_{1} \leqslant \ldots \leqslant V_{n} = B(p)
$$
of~the~same length that satisfy the~relations
\begin{align*}
U_{i} \cap H(p) = V_{i} \cap H(p),\quad 0 &\leqslant i \leqslant n,\\
|U_{i+1} \cap H(p)/U_{i} \cap H(p)| = p,\quad 0 &\leqslant i \leqslant n-1.
\end{align*}
Let us put $R_{i} = U_{i}C$, $S_{i} = V_{i}D$, and~$Q_{i} = U_{i} \cap H(p)$, where $0 \leqslant i \leqslant n$.

Since $C \in \mathcal{F}_{p}^{*}(A)$ and~$D \in \mathcal{F}_{p}^{*}(B)$, we have $R_{i}^{\vphantom{*}} \in \mathcal{F}_{p}^{*}(A)$ and~$S_{i}^{\vphantom{*}} \in \mathcal{F}_{p}^{*}(B)$. It~follows from~the~equalities $C \cap A(p) = 1 = D \cap B(p)$ and~the~inclusions $U_{i} \leqslant A(p)$, $V_{i} \leqslant B(p)$, and~$H(p) \leqslant A(p) \cap B(p)$ that
$$
R_{i} \cap H(p) = U_{i} \cap H(p) = Q_{i} = V_{i} \cap H(p) = S_{i} \cap H(p), \quad 0 \leqslant i \leqslant n.
$$
Thus, the~sequences
\begin{gather*}
1 = Q_{0} \leqslant Q_{1} \leqslant \ldots \leqslant Q_{n} = H(p),\\
R_{0} \leqslant R_{1} \leqslant \ldots \leqslant R_{n} = A,
\quad\text{and}\quad
S_{0} \leqslant S_{1} \leqslant \ldots \leqslant S_{n} = B
\end{gather*}
satisfy all the~conditions of~Theorem~\ref{et03}.
\end{proof}

\end{document}